\documentclass[onefignum,onetabnum]{siamart190516}

\usepackage[caption=false]{subfig}
\usepackage{pgfplots}
\usepackage{lipsum}
\usepackage{amsfonts}
\usepackage{graphicx}
\usepackage{epstopdf}
\usepackage{algorithmic,cite}
\usepackage{microtype}
\ifpdf
  \DeclareGraphicsExtensions{.eps,.pdf,.png,.jpg}
\else
  \DeclareGraphicsExtensions{.eps}
\fi

\newsiamremark{remark}{Remark}
\newsiamremark{hypothesis}{Hypothesis}
\crefname{hypothesis}{Hypothesis}{Hypotheses}
\newsiamthm{assumption}{Assumption}

\headers{Convergence rate analysis of SCP$_{\rm ls}$}{Peiran Yu, Ting Kei Pong, and Zhaosong Lu}

\title{Convergence rate analysis of a sequential convex programming method with line search for a class of constrained difference-of-convex optimization problems
\thanks{\funding{The second author was supported partly by Hong Kong Research Grants Council PolyU153005/17p.}}}

\author{Peiran Yu\thanks{Department of Applied Mathematics, The Hong Kong Polytechnic University, Hong Kong
  (\email{peiran.yu@connect.polyu.hk}).}
\and Ting Kei Pong\thanks{Department of Applied Mathematics, The Hong Kong Polytechnic University, Hong Kong
  (\email{tk.pong@polyu.edu.hk}).}
\and Zhaosong Lu\thanks{Department of Industrial and Systems Engineering, University of Minnesota, USA
  (\email{zhaosong@umn.edu}).}
}

\ifpdf
\hypersetup{
  pdftitle={Convergence  analysis on the sequential convex programming  method with monotone line search for multiple constrained difference-of-convex  model},
  pdfauthor={Peiran Yu and Ting Kei Pong and Zhaosong Lu}
}
\fi

\usepackage{amsopn}

\usepackage{amssymb,latexsym,amsmath,enumerate}
\usepackage{bm,setspace}
\allowdisplaybreaks[4]

\newcommand{\revise}[1]{\textcolor{black}{\ignorespaces#1\ignorespaces}}   
\newcommand{\rerevise}[1]{\textcolor{black}{\ignorespaces#1\ignorespaces}}   

\def\R{{\rm I\!R}}
\def\N{{\rm I\!N}}
\def\argmin{\mathop{\rm arg\,min}}
\def\Argmin{\mathop{\rm Arg\,min}}

\begin{document}
\maketitle
\begin{abstract}
In this paper, we study the sequential convex programming method with monotone line search (SCP$_{ls}$) in \cite{Lu10} for a class of difference-of-convex (DC) optimization problems with multiple smooth inequality constraints. The SCP$_{ls}$ is a representative variant of moving-ball-approximation-type algorithms \cite{AuShTe10,BoPa16,BoChPa19,ShTe16} for constrained optimization problems. We analyze the convergence rate of the sequence generated by SCP$_{ls}$ in both nonconvex and convex settings by imposing suitable Kurdyka-{\L}ojasiewicz (KL) assumptions. Specifically, in the nonconvex settings, we assume that a special potential function related to the objective and the constraints is a KL function, while in the convex settings we impose KL assumptions directly on the extended objective function (i.e., sum of the objective and the indicator function of the constraint set). A relationship between these two different KL assumptions is established in the convex settings under additional differentiability assumptions. We also discuss how to deduce the KL exponent of the extended objective function from its Lagrangian in the convex settings, under additional assumptions on the constraint functions. Thanks to this result, the extended objectives of some constrained optimization models such as minimizing $\ell_1$ subject to logistic/Poisson loss are found to be KL functions with exponent $\frac12$ under mild assumptions. To illustrate how our results can be applied, we consider SCP$_{ls}$ for minimizing $\ell_{1-2}$ \cite{YiLoHeXi15} subject to residual error measured by  $\ell_2$ norm/Lorentzian norm \cite{CaRaArBaSa16}. We first discuss how the various conditions required in our analysis can be verified, and then perform numerical experiments to illustrate the convergence behaviors of SCP$_{ls}$.
\end{abstract}

 \section{Introduction}
Constrained optimization problems naturally arise when one attempts to find a solution that minimizes a certain objective under some restrictions, see \cite{vaFr09,AuShTe10,JeYaZh11,BoVa04,CaRaArBaSa16}.
Here, we consider the following specific type of difference-of-convex (DC) constrained optimization problem:
  \begin{equation}\label{Or}
    \begin{split}
    \min_{x\in\R^n} &\    F(x):=f(x) + P_1(x) - P_2(x) +\delta_{g(\cdot)\leq 0} (x),
    \end{split}
  \end{equation}
where $f:\R^n\to \R$ is smooth, $P_1:\R^n\to\R$ and $P_2:\R^n\to\R$ are convex continuous \rerevise{(possibly nonsmooth)}, and $g:\R^n\to \R^m$ is continuous with $\{x:\;g(x)\le0\}\neq\emptyset$.
In typically applications, the $f$ in \eqref{Or} arises as measures for data fidelity, $g$ is used for modeling restrictions on the decision variable $x$, and $P_1-P_2$ is a regularizer for inducing desirable structures; see \cite[Table~1]{GZLHY13} for examples of such regularizers.
In our subsequent algorithmic development for \eqref{Or}, we also consider the following additional assumption.
\begin{assumption}\label{basebase}
    Let $f$, $g$ and $F$ be as in \eqref{Or}.\footnote{\rerevise{Here and throughout the paper, by referring to \eqref{Or}, we mean that the assumptions on $f$, $P_1$, $P_2$ and $g$ stated right after \eqref{Or} are satisfied, i.e., $f:\R^n\to \R$ is smooth, $P_1:\R^n\to\R$ and $P_2:\R^n\to\R$ are convex continuous, and $g:\R^n\to \R^m$ is continuous with $\{x:\;g(x)\le0\}\neq\emptyset$.}}
    \begin{enumerate}[{\rm (i)}]
      \item \revise{ $f:\R^n\to\R$ } has Lipschitz continuous gradient with Lipschitz modulus $L_f$.
      \item \rerevise{For the mapping $g(x)=(g_1(x),...,g_m(x))$, each function $g_i$ has Lipschitz} continuous gradient with Lipschitz modulus $L_{g_i}$.
      \item The function $F$ is level-bounded.
    \end{enumerate}
\end{assumption}
Under Assumption~\ref{basebase}, the solution set of \eqref{Or} is nonempty and $\inf F>-\infty$.

To design algorithms for solving \eqref{Or} under Assumption~\ref{basebase}, one common approach is to resort to the majorization-minimization (MM) procedure: in this procedure, one iteratively constructs and minimizes a surrogate function that locally majorizes $F$; \rerevise{see \cite{BoPa16,DrIo19,DrPa19,DrLe18,LeWr16,SuBaPa17} for related models and discussions}.
For \eqref{Or} under Assumption~\ref{basebase}, one natural way to construct surrogate function is to make use of the 2nd-order Taylor's expansions of $f$ and $g$: the resulting algorithms are the moving balls approximation method (MBA) proposed in \cite{AuShTe10} (for $P_1 = P_2 = 0$) and its variants \cite{BoChPa19,BoPa16}. In each iteration, these algorithms approximate the constraint $g(x)\le 0$ in \eqref{Or} by
       \begin{align}\label{barG}
       \bar  G(x,y,w): =\begin{pmatrix}
   g_1(y) + \langle\nabla g_1(y),x-y\rangle + \frac{w_1}2\|x-y \|^2\\
   \vdots\\
   g_m(y) + \langle\nabla  g_m(y),x-y\rangle + \frac{w_m}2\|x-y \|^2
 \end{pmatrix}\le 0
       \end{align}
for some fixed $(y,w)$: the feasible region of the resulting subproblem is an intersection of $m$ balls.
For the sequence generated by MBA, global convergence to a minimizer was established in \cite{AuShTe10} when $\{f,g_1,\dots,g_m\}$ are in addition convex and the Slater condition holds. The linear convergence of the sequence generated by MBA was also proved in \cite{AuShTe10} when $f$ in \eqref{Or} is additionally  strongly convex. \revise{In \cite{BoPa16}, when $\{f,g_1,\dots,g_m\}$ are semi-algebraic and $P_1 = P_2=0$ in \eqref{Or},  the whole sequence generated by an MBA variant was shown to converge to a critical point and its convergence rate was also established, under the Mangasarian-Fromovitz constraint qualification (MFCQ). }

\revise{\rerevise{When the  DC function $P_1-P_2$ in \eqref{Or} is nonsmooth (these nonsmooth functions arise naturally as regularizers in applications such as sparse recovery \cite{CaRaArBaSa16,GZLHY13,YiLoHeXi15}), the MBA method is not directly applicable. Moreover, when $P_2$ is nonsmooth,
 the multiprox method in \cite{BoChPa19} and the majorization-minimization procedure in \cite[Section~3]{BoPa16} cannot be directly applied to \eqref{Or}.}  Fortunately, under Assumption~\ref{basebase}, problem \eqref{Or} has DC objective and DC constraints: indeed, one can write $f$ and each $g_i$ in \eqref{Or} as the difference of two convex functions as follows:
\[
f(x) = \frac{L_f}{2}\|x\|^2 - \left(\frac{L_f}{2}\|x\|^2 - f(x)\right)\ {\rm and}\ g_i(x) = \frac{ L_{g_i}}{2}\|x\|^2 - \left(\frac{ L_{g_i}}{2}\|x\|^2 - g_i(x)\right).
\]
DC algorithms (DCA) (see, for example, \cite{LeTa18,LeNgTa04}) can thus be applied. A variant that specializes in functional constraints is the sequential convex programming (SCP) method proposed in \cite{Lu10} \footnote{\revise{We would like to point out that the methods proposed in \cite{Lu10} (including SCP and its variant) were designed to solve more general models than \eqref{Or}. In particular, they can deal with problems with \rerevise{constraints involving nonsmooth functions}, and allow for nonmonotone line search.}}; see also \cite[Remark~5]{QuDi11}. When applied to \eqref{Or} under Assumption~\ref{basebase}, this method maintains feasibility at each iteration\footnote{\revise{There       are some DCA variants for solving \eqref{Or} under Assumption~\ref{basebase} that do not maintain feasibility throughout. We refer the interested readers to \cite{LeTaNg12,LeNgTa04,LiBo16,StMi18,TaAn14} for more discussions.}} and each subproblem is constrained over an intersection of balls: thus, this method can also be viewed as a variant of MBA. It was  shown that any accumulation point of the sequence generated by SCP is a stationary point under Slater's condition. However, convergence and convergence rate   of the whole sequence generated remain unknown.\footnote{\revise{We point out that convergence of the whole sequence and the convergence rate generated by some DCA variants were considered in \cite{ArVu20,LeTa18} under suitable Kurdyka-{\L}ojasiewicz (KL) assumptions; however, their problem formulations do not explicitly involve functional constraints as in \eqref{Or}.}}
}

For empirical acceleration, a variant of MBA that involves a line search scheme was proposed in \cite{BoChPa19}, which is called the Multiproximal method with backtracking step sizes (Multiprox$_{bt}$). When applied to \eqref{Or} under Assumption~\ref{basebase}, the sequence generated by Multiprox$_{bt}$ converges to a minimizer when $\{f,g_1,\dots,g_m\}$ are additionally convex, $P_1=P_2=0$ and the Slater condition holds. However, Multiprox$_{bt}$ uses monotone initial step sizes, \revise{i.e., $\tilde \alpha$ in \cite[Eq.~(37)]{BoChPa19} is nondecreasing as the algorithm progresses,} which rules out widely used choices such as the truncated \revise{Barzilai}-Borwein step sizes \cite{BaBo88,BirMarRay00}.
\revise{On the other hand, the line search variant of SCP proposed in \cite{Lu10} can incorporate flexible line search schemes like the truncated \revise{Barzilai}-Borwein step size and is general enough to be applied to \eqref{Or} under Assumption~\ref{basebase} with possibly nonsmooth $P_1-P_2$.
In \cite{Lu10}, the well-definedness of the proposed algorithm was established, and it was also shown that any accumulation point of it is a stationary point under Slater's condition. However, convergence of the whole sequence generated and the corresponding convergence rate is still open.}

\revise{In this paper, we further study the line search variant of the SCP method  proposed in \cite{Lu10} with its line search being monotone, i.e., $M$ in \cite[Eq.~(22)]{Lu10} being $0$. We call this variant SCP$_{ls}$; see Algorithm~\ref{alg1} below.
We analyze the convergence properties of the sequence generated by SCP$_{ls}$ for solving \eqref{Or} under Assumption~\ref{basebase}. }The main convergence rate analysis of SCP$_{ls}$ is presented in Section~\ref{sec3}. We derive global convergence rate of the sequence generated by SCP$_{ls}$ in the following two scenarios:
\begin{itemize}
          \item {\em \revise{ $F$ in \eqref{Or}} is possibly  nonconvex with each $g_i$ being twice continuously differentiable and $P_2$ being Lipschitz continuously differentiable  on an open set $\Gamma$ that contains the set of stationary points of $F$.}

              Our analysis is based on the following specially constructed potential function:
\begin{align}\label{barF}
\bar F(x,y,w) = f(x) + P_1(x) - P_2(x) + \delta_{\bar  G(\cdot)\le 0}(x,y,w),
\end{align}
where $\bar G$ is defined  as in \eqref{barG}. Under MFCQ, we characterize the local convergence rate of the sequence generated by SCP$_{ls}$ according to the Kurdyka-{\L}ojasiewicz (KL) exponent of $\bar F$.  Note the mapping  $(x,y) \mapsto \bar F(x,y,L)$ with \rerevise{$P_2 = 0$} and $L$ being a constant positive vector (related to the step size)  was used previously in \cite{BoPa16} for establishing the convergence of an MBA variant when  \rerevise{$P_2 = 0$} and $\{f,g_1,\dots,g_m\}$ in \eqref{Or} are semi-algebraic. \revise{ This kind of potential functions was called ``value function" in \cite{Pa16} and was used there for deducing the global convergence properties  of the composite Gauss-Newton method for composite optimization problems. Our potential function $\bar F$ allows us to deal with more flexible stepsize rules than those studied in \cite{Pa16, BoPa16}.}

          \item {\em \revise{$\{f,g_1\dots,g_m\}$ in \eqref{Or}} are convex and $P_2= 0$.}

          \revise{This same convex setting was considered in \cite[Section 3.2.3]{BoChPa19}.
          In this setting,} we impose KL assumptions directly on $F$ in \eqref{Or} (instead of on $\bar F$). In particular, a local {\em linear} convergence rate is established when $F$ is a KL function with exponent $\frac12$, under MFCQ. This is different from many existing analysis (see, for example, \cite{AtBoReSo10,BoPa16,LiPong16,Ochs19}), which typically make use of the KL property of a potential function constructed out of $F$ instead of $F$ itself.
\end{itemize}

        In Section~\ref{KLofFconvex}, we study a relationship between the KL property of $\bar F$ in \eqref{barF} and that of $F$ in \eqref{Or}. Then, we present a ``calculus rule" that deduces  the KL exponent of \revise{ $F$ in \eqref{Or} }from its Lagrangian in the convex settings, under some mild assumptions. This enables us to deduce that \revise{the function $F$ corresponding} to minimizing  $\ell_1$ subject to logistic/Poisson loss is a KL function with exponent $\frac12$ under mild conditions.

         In Section~\ref{applications}, we  discuss some concrete models to which SCP$_{ls}$ can be applied. Specifically, we consider models of the following form:\vspace{-0.1cm}
         \begin{equation}\label{introapp}\vspace{-0.1cm}
         \begin{array}{rl}
         \min\limits_{x} &\ \|x\|_1 - \mu\|x\|\\
         {\rm s.t.}&\ \ell(Ax-b)\le\delta,
         \end{array}
          \end{equation}
         where $\mu\in[0,1]$, $A\in\R^{q\times n}$ has {\em full row rank}, $b\in\R^q$, $\ell:\R^q\to \R_+$ is analytic with Lipschitz continuous gradient and satisfies $\ell(0)= 0$, and $\delta \in (0,\ell(-b))$. This model arises in compressed sensing where the measurements $b$ may be corrupted by  different types of noise; see \cite{CaBaAy10}. We focus on two concrete choices of $\ell$: the square of norm (for noise following Gaussian distribution) and the Lorentzian norm (for noise following Cauchy distribution). For these two choices, we provide suitable conditions on the problem data so that the assumptions in our convergence results are satisfied. Then we perform numerical tests on solving \eqref{introapp} with $\ell$ being either the square of norm or the Lorentzian norm via \revise{ two methods: SCP$_{ls}$ and SCP \cite{Lu10}.} We observe that SCP$_{ls}$ appears to converge linearly and \revise{is much faster}.

\section{Notation and preliminaries}
In this paper, we let $\R$ denote the set of real numbers and $\N_+$ denote the set of positive integers. The $n$-dimensional Euclidean space is denoted by $\R^n$, and the nonnegative orthant is denoted by $\R^n_+$. For two vectors $x$ and $y\in \R^n$, we write $x\ge y$ if $x_i\ge y_i$ for all $i$. The Euclidean norm of $x$ is denoted by $\|x\|$, the inner product of $x$ and $y$ is denoted by $\left<x,y\right>$, and the $\ell_1$ norm of $x$ is denoted by $\|x\|_1$. For $x\in \R^n$ and \rerevise{$r\ge 0$}, we let $B(x,r)$ denote the closed ball centered at $x$ with radius $r$, i.e., $B(x,r) = \{y:\; \|x - y\|\le r\}$.

We say that an extended-real-valued function $f: \R^{n}\rightarrow (-\infty,\infty]$ is proper if its domain ${\rm dom} f:=\{x:\; f(x)<\infty\}\neq \emptyset$. A proper function $f$ is said to be closed if it is lower semicontinuous. For a proper function $f$, the regular subdifferential of $f$ at $x\in {\rm dom} f$ is defined by\vspace{-0.1cm}
\[\vspace{-0.1cm}
\hat{\partial} f(x):=\bigg\{\zeta :\; \liminf\limits_{z\rightarrow x,z\neq x}\frac{f(z)-f(x)-\left<\zeta,z-x\right>}{\|z-x\|}\ge 0 \bigg\}.
\]
The (limiting) subdifferential of $f$ at $x\in {\rm dom} f$  is defined by\vspace{-0.1cm}
\begin{equation*}\vspace{-0.1cm}
\partial f(x):=\bigg\{\zeta :\; \exists x^{k}\stackrel{f}{\rightarrow} x, \zeta^{k}\rightarrow \zeta \ {\rm with\ } \zeta^k \in \hat{\partial} f(x^k)\  {\rm for\  each}\  k \bigg\},
\end{equation*}
where $x^{k}\stackrel{f}{\rightarrow} x$ means both $x^{k}\to x$ and $f(x^{k})\to f(x)$. Moreover, we set $\partial f(x) = \hat \partial f(x) = \emptyset$ for $x\notin {\rm dom}\, f$ by convention, and we write ${\rm dom}\,\partial f:= \{x:\; \partial f(x)\neq \emptyset\}$.
\revise{When $f$ is proper convex, thanks to \cite[Proposition~8.12]{Lu10}, the limiting subdifferential and regular subdifferential  of $f$ at an $x\in{\rm dom}\, f$ reduce to the classical subdifferential}, which is given by
\[
\partial f(x)=\{\zeta:\; \left<\zeta,y- x\right>\le f(y) - f(x)\ \rerevise{{\rm for\ all \ }y}\}.
\]
For a nonempty set $C$, the indicator function $\delta_C$ is defined as
\[
\delta_C(x):=
\begin{cases}
0&  x\in C,\\
\infty & x\notin C.
\end{cases}
\]
The normal cone (resp., regular normal cone) of $C$ at an $x\in C$ is defined as $N_C(x):=\partial \delta_C(x)$ (resp., $\hat N_C(x) := \hat\partial \delta_C(x)$), and the distance from a point $x\in \R^n$ to $C$ is denoted by ${\rm dist}(x,C)$.

We next recall the KL property and the notion of KL exponent; see \cite{Loja63,Kur98,AtBo09,AtBoReSo10,AtBoSv13,LiPo18}. This property has been used extensively for analyzing convergence properties  of first-order methods; see, for example, \cite{AtBo09,AtBoReSo10,AtBoSv13,BoSaTe14,WeChPo18}.

\begin{definition}[{{\bf Kurdyka-{\L}ojasiewicz property and exponent}}]
  We say that a proper closed function $h:\R^n\to(-\infty,\infty]$ satisfies the Kurdyka-{\L}ojasiewicz (KL) property at an $\hat x\in {\rm dom} \partial h$ if there are $a\in (0,\infty]$, a neighborhood $V$ of $\hat{x}$ and a continuous concave function $\varphi: [0,a)\rightarrow [0,\infty) $ with $\varphi(0)=0$ such that
  \begin{enumerate}[{\rm (i)}]
    \item $\varphi$ is continuously differentiable on $(0, a)$ with $\varphi'>0$ on $(0,a)$;
    \item for any $x\in V$ with $h(\hat{x})<h(x)<h(\hat{x})+a$, it holds that
     \begin{equation}\label{phichoice}
     \varphi'(h(x)-h(\hat{x})){\rm dist}(0,\partial h(x))\ge 1.
     \end{equation}
  \end{enumerate}
  If $h$ satisfies the KL property at $\hat x\in {\rm dom}\partial h$ and \revise{ $\varphi$ in \eqref{phichoice}} can be chosen as $\varphi(\nu) = a_0\nu^{1-\alpha}$ for some $a_0>0$ and $\alpha\in[0,1)$, then we say that $h$ satisfies the KL property at $\hat x$ with exponent $\alpha$.

  A proper closed function $h$ satisfying the KL property at every point in ${\rm dom} \partial h$ is called a KL function, and a proper closed function $h$ satisfying the KL property with exponent $\alpha\in [0,1)$ at every point in ${\rm dom} \partial h$ is called a KL function with exponent $\alpha$.
\end{definition}

There are many examples of KL functions. For instance,  proper closed semi-algebraic functions and proper subanalytic functions that have closed domains and are continuous on their domains are \rerevise{KL functions}; see \cite{AtBoReSo10} and  \cite[Theorem~3.1]{BoDaLe07}, respectively.

Now we recall the definition of stationary points of \eqref{Or} when $g_i$ are smooth.
\begin{definition}[\textbf{Stationary point}]\label{def:stat}
Consider \eqref{Or} and assume that each $g_i$ is smooth. We say that an $x\in\R^n$ is a stationary point of \eqref{Or} if there exists $\lambda\in\R^m_+$ such that $(x,\lambda)$ satisfies\vspace{-0.1 cm}
 \begin{equation*}\vspace{-0.1 cm}
 g(x)\le0,\ \lambda_i g_i(x) = 0 {\rm \ for\ all\ } i, {\rm \ and \ }0\in  \nabla f(x)  + \partial P_1(x) - \partial P_2(x) +  \sum_{i=1}^m\lambda_i\nabla g_i(x).
  \end{equation*}
\end{definition}

The following assumption will be used repeatedly throughout this paper.
\begin{assumption}\label{baseass}
Each $g_i$ in \eqref{Or} is smooth and the Mangasarian-Fromovitz constraint qualification (MFCQ) holds  in the whole domain of $F$ in \eqref{Or}, i.e., for every $x$ satisfying $g(x)\le 0$, there exists $d\in\R^n$ such that
           \[
           \left<\nabla g_i(x),d\right><0{\rm \ for \  each\  }i\in I(x):=\{j:\;\;g_j(x)=0\}.
           \]
\end{assumption}

Under Assumptions~\ref{basebase} and \ref{baseass}, it is routine to show that any local minimizer of \eqref{Or} is a stationary point in the sense of Definition~\ref{def:stat}.
 In fact, let $\hat x$ be a local minimizer of \eqref{Or}. Using \cite[Theorem~10.1]{RoWe97}, we have
\begin{equation}\label{sta1}
  \begin{split}
  &0\in\partial F(\hat x) \stackrel{{\rm (a)}}\subseteq\nabla f(\hat x) + \partial P_1 (\hat x) + \partial(-P_2)(\hat x) + \partial  \delta_{g(\cdot)\le 0}(\hat x)\\
  &\stackrel{{\rm (b)}}\subseteq  \nabla f(\hat x) + \partial P_1 (\hat x) +\bar \partial( - P_2)(\hat x) + \partial  \delta_{g(\cdot)\le 0}(\hat x)\\
  &\stackrel{{\rm (c)}}{=}  \nabla f(\hat x) + \partial P_1 (\hat x) -\bar \partial P_2(\hat x) + \partial  \delta_{g(\cdot)\le 0}(\hat x)\\
  &= \nabla f(\hat x) + \partial P_1 (\hat x) -\partial P_2(\hat x) + \partial  \delta_{g(\cdot)\le 0}(\hat x),
  \end{split}
\end{equation}
where (a)  follows from \cite[Exercise~10.10]{RoWe97}, the inclusion (b) uses \cite[Theorem~5.2.22]{BoZh05}, where $\bar \partial(-P_2)$ is the Clarke subdifferential of $-P_2$, the equality (c) uses \cite[Proposition~2.3.1]{Clarke90} and  the last equality  holds because of  the convexity of $P_2$ and \cite[Theorem~6.2.2]{BoLe06}.
In addition, we can deduce that
\begin{align*}
&\partial  \delta_{g(\cdot)\le 0}(\hat x) = N_{g(\cdot)\le 0}(\hat x) =  \left\{\sum_{i=1}^m\lambda_i\nabla g_i(\hat x): \lambda\in N_{-\R^m_+}(g(\hat x))\right\} \\
&= \left\{\sum_{i=1}^m\lambda_i\nabla g_i(\hat x): \lambda\in\R^m_+, \ \lambda_ig_i(\hat x) = 0\ {\rm for }\ i=1,\dots,m\right\},
\end{align*}
where the second equality follows from MFCQ and \cite[Theorem~6.14]{RoWe97} and the last equality follows from the definition of normal cone. The above display together with \eqref{sta1} shows that $\hat x$ is a stationary point of \eqref{Or}. \revise{In passing, we would like to point out that $x^*$ is a stationary point of \eqref{Or} in the sense of Definition~\ref{def:stat} if and only if there exists $\xi^*$ such that $0 \in \partial \widetilde F(x^*,\xi^*)$, where $\widetilde F(x,\xi):= f(x) + P_1(x) - \left<\xi,x\right> + P^*_2(\xi) +\delta_{g(\cdot)\le 0}(x)$, with $\{P_1,P_2\}$ given in \eqref{Or} and $P^*_2$ being the Fenchel conjugate of $P_2$. This type of stationary points is widely used in the DC literature; see, for example, \cite{TaAn97,TaAn14,WeChPo18}. Note that there are other concepts of stationarity used in the literature, such as the Clarke stationarity, d-stationarity and B-stationarity; we refer the readers to \cite{PaRaAl17,vade19,JoBaKaMaTa18} for more discussions. The notion of stationarity defined in Definition \ref{def:stat} is in general weaker than these aforementioned notions.   }

Before ending this section, we introduce the algorithm we analyze and present some auxiliary results for our subsequent analysis.
The algorithm, SCP$_{ls}$ proposed in \cite{Lu10}, is presented in Algorithm~2.1, where $\bar G$ is defined as in \eqref{barG}. Notice that by rearranging terms of the constraint functions of the subproblem \eqref{subp}, we can see that the constraint there is equivalent to
\begin{align}\label{rit}
x\in \bigcap_{i=1}^mB\left(\widetilde s_i,\sqrt{\widetilde R_i}\right),
\end{align}
where $\widetilde s_i:=x^t - \frac{1}{\left(\widetilde L_{g}\right)_{i}}\nabla g_i(x^t)$
and $\widetilde R_i:=\left\|\frac{\nabla g_i(x^t) }{\left(\widetilde L_{g}\right)_{i}}\right\|^2 - \frac{2}{\left(\widetilde L_{g}\right)_{i}}g_i(x^t)$.
Thus, when $m = 1$, the constraint reduces to a {\em single} ball constraint and a simple root-finding scheme was discussed in \cite{ShTe16} for exactly and efficiently solving \revise{ the subproblem \eqref{subp}} with $m = 1$, $P_2= 0$ and $P_1$ being the $\ell_1$ norm or the nuclear norm, etc. However, solving subproblem \eqref{subp} in general requires an iterative solver; see \cite[Section~6]{AuShTe10} for the case when $P_1=P_2= 0$.

\begin{algorithm}
\caption{ Sequential  convex programming  method with monotone line search (SCP$_{ls}$) for \eqref{Or} under Assumption~\ref{basebase}\label{alg1}}
\begin{algorithmic}
\STATE  Choose parameters $c>0$, $0<\b{L}<{\rm \bar L} $, $\tau>1$ and an $x^0$ with $g(x^0)\le 0$. Set $t=0$.
\begin{description}
\item[\bf Step 1.] Pick any $\xi^t \in \partial P_2(x^t)$.

\item[\bf Step 2.] Choose $L_f^{t,0}\!\!\in\!\![\b{L}, {\rm \bar L} ]$  and $ L_g^{t,0}\!\!\in\!\![\b L, {\rm \bar L}]^m$ arbitrarily. Set $\widetilde L_f \!\!= \!\!L_f^{t,0}$ and $\widetilde L_{g} \!\!=\!\!L_{g}^{t,0}$.

\item[\bf Step 3.] Compute
\begin{equation}\label{subp}
\begin{split}
\widetilde x =\argmin_{x}&\left\{\langle\nabla f(x^t) - \xi^t,x - x^t\rangle + \frac{\widetilde L_f}{2}\|x - x^t\|^2 + P_1(x)\right\}\\
 {\rm s.t.}& \ \  \bar G(x,x^t,\widetilde L_{g})\le 0.
\end{split}
\end{equation}
\begin{description}
  \item[\bf Step 3a)]  If $g(\widetilde x)\le 0$ and
\begin{align}\label{decrease}
F(\widetilde x)\leq  F(x^t) - \frac{c}{2}\|\widetilde x - x^t\|^2
\end{align}
holds, go to step 4.
  \item[\bf Step 3b)] If $g(\widetilde x)\not\le 0$, let $\widetilde L_{g} \leftarrow \tau\widetilde L_{g}$ and go to step 3.
  \item[\bf Step 3c)] If \eqref{decrease} does not hold,
let $\widetilde L_{f} \leftarrow \tau \widetilde L_{f}$ and go to step 3.
\end{description}
\item[\bf Step 4.] If a termination criterion is not met, set $L_{g}^t = \widetilde L_{g}$, $L_f^t = \widetilde L_f$ and $x^{t+1} = \widetilde x$. Update $t\leftarrow t+1$ and go to {\bf Step 1.}
    \end{description}
\end{algorithmic}
\end{algorithm}

In the next lemma, we discuss the well-definedness of SCP$_{ls}$ and also establish some inequalities needed in our analysis below.
Note that the well-definedness of SCP$_{ls}$ was already proved in \cite[Theorem~3.6]{Lu10} in a more general setting. Here we include its proof for completeness.
\begin{lemma}\label{rdayu0}
    Consider \eqref{Or} and suppose that Assumptions~\ref{basebase} and \ref{baseass} hold. Then the following statements hold:
    \begin{enumerate}[{\rm(i)}]
    \item SCP$_{ls}$ is well defined, i.e., the subproblems \eqref{subp} are well defined and there exists a $k_0\in \mathbb{N}_+$ (independent of $t$) such that in any iteration $t\ge 0$, the inner loop stops after at most $k_0$ iterations.
    \item The sequence $\{(L_f^t,L_g^t)\}$ generated by SCP$_{ls}$ is bounded.
    \item For each $i\in\{1,\dots,m\}$, each $t\ge 0$ and each $(\widetilde L_f,\widetilde L_g)$, the $\widetilde R_i$ in \eqref{rit} is positive.
    \item For each $t\ge 0$ and each $(\widetilde L_f,\widetilde L_g)$, the problem \eqref{subp} has  a Lagrange multiplier $\widetilde \lambda$. Let $\widetilde L_{fg}: =\widetilde L_f +  \langle\widetilde \lambda, \widetilde L_{g}\rangle$ and let $\widetilde x$ be as in \eqref{subp}.  Then
\begin{align}\label{lamg0}
 \widetilde \lambda_i \left(g_i(x^t) + \langle\nabla g_i(x^t),\widetilde x-x^t\rangle + \frac{( \widetilde L_{g})_{i}}2\|\widetilde x-x^t \|^2 \right)= 0 {\rm \ for \ all\ }i,
\end{align}
and
 \begin{equation}\label{firstoderofsubp}
  0\in \nabla f(x^t) - \xi^t + \widetilde L_{fg}(\widetilde x - x^t) + \partial P_1(\widetilde x) + \sum_{i=1}^m\widetilde \lambda_i\nabla g_i(x^t),
  \end{equation}
  where $\{x^t\}$ and $\{\xi^t\}$ are generated by SCP$_{ls}$.
  Moreover, if $g(\widetilde x)\le 0$, then for any $x\in \R^n$ we have
  \begin{equation}\label{strongc}
  \begin{split}
&F(\widetilde x)\!\le\! f(x^t) \!+ \!\left<\nabla f(x^t) - \xi^t,x\! -\!x^t\right>\! +\! \frac{\widetilde L_{fg}}{2}\|x \!-\!x^t\|^2 \!+\! P_1(x) \!-\! P_2(x^{t}) \\
     &+\!\sum_{i=1}^m\widetilde \lambda_i \left(g_i(x^t) + \langle\nabla g_i(x^t), x-x^t\rangle  \right)\!-\!\frac{\widetilde L_{fg}}{2}\|x - \widetilde x\|^2\! - \!\frac{\widetilde L_f - L_f}{2}\|\widetilde x -x^t\|^2. \end{split}
  \end{equation}
\end{enumerate}
\end{lemma}
\begin{proof}
Let an $x^t$ satisfying $g(x^t)\le 0$ be given for some $t\ge 0$. We will first show that the corresponding subproblems \eqref{subp} are well defined (for any $(\widetilde L_f,\widetilde L_g)$) and the conclusions of items (iii) and (iv) hold for this $t$. Using these, we will then show that there exists $k_0$ (independent of $t$) so that the inner loop in Step 3 terminates after $k_0$ iterations and returns an $x^{t+1}$ that satisfies $g(x^{t+1})\le 0$. This together with $g(x^0)\le 0$ and an induction argument will show that SCP$_{ls}$ is well defined and that items (iii) and (iv) hold for all $t\ge 0$. Finally, we show that $\{(L^t_f,L^t_g)\}$ is bounded.

Suppose that an $x^t$ satisfying $g(x^t)\le 0$ is given for some $t\ge 0$. Notice that for any $(\widetilde L_f,\widetilde L_g)$, the feasible region of \eqref{subp} is nonempty (it contains $x^t$) and the subproblem is to minimize a strongly convex continuous function over a nonempty closed convex set. Thus, $\widetilde x$ exists and is unique.
Now, fix any $i\in\{1,\dots,m\}$. Since $g(x^t)\le 0$ and $(\widetilde L_{g})_i> 0$, we have $- \frac{2}{(\widetilde L_{g})_i}g_i(x^t)\geq 0$ and thus $\widetilde R_i\ge 0$. Suppose to  the contrary that $\widetilde R_i = 0$. Then we have $\nabla g_i(x^t) = 0$ and $g_i(x^t) = 0$,
contradicting Assumption~\ref{baseass}. Thus, we must have $\widetilde R_i>0$ at the $t^{\rm th}$ iteration.

Next, using a similar proof of \cite[Proposition~2.1(iii)]{AuShTe10}, we deduce using MFCQ that the Slater condition holds for \eqref{subp} for this $t$. Therefore, using \cite[Corollary~28.2.1, Theorem~28.3]{Ro70}, for problem \eqref{subp}, there exists a Lagrange multiplier  $\widetilde \lambda\in\R_+^m$ such that
\eqref{lamg0} holds at the $t^{\rm th}$ iteration and $\widetilde x$ is a minimizer of the following function:
\begin{align*}
L_t(x,\widetilde\lambda):= &f(x^t) + \left<\nabla f(x^t),x -x^t\right> + \frac{\widetilde L_f}{2}\|x -x^t\|^2 + P_1(x)- P_2(x^{t})   \\
  &- \left<\xi^t,x - x^t\right> +\langle \widetilde \lambda,\bar G(x,x^t,\widetilde L_g)\rangle.
\end{align*}
This together with \cite[Theorem~10.1, Exercise~8.8]{RoWe97} shows that \eqref{firstoderofsubp} holds at the $t^{\rm th}$ iteration.

In addition, note that $x\mapsto L_t(x,\widetilde\lambda)$ is strongly convex with modulus $\widetilde L_{fg}$. Then we see that for any $x\in \R^n$,
\begin{equation}\label{strongcpre}
\begin{aligned}
  &f(x^t) \!+\! \left<\nabla f(x^t),\widetilde x \!-\!x^t\right> \!+\! \frac{\widetilde L_f}{2}\|\widetilde x \!-\!x^t\|^2 \!+\! P_1(\widetilde x)\!-\! P_2(x^{t}) \!-\! \left<\xi^t,\widetilde x \!-\! x^t\right>\\
  &=L_t(\widetilde x,\widetilde \lambda)  \le L_t(x,\widetilde \lambda) \!-\! \frac{\widetilde L_{fg}}{2}\|x \!-\! \widetilde x\|^2\\
  & = f(x^t) \!+\! \left<\nabla f(x^t), x \!-\!x^t\right> \!+\! \frac{\widetilde L_{fg}}{2}\|x \!-\!x^t\|^2 \!+\! P_1( x)\!-\! P_2(x^{t}) \!-\! \left<\xi^t, x \!-\! x^t\right>  \\
  &\ \ \ \ +\!\sum_{i=1}^m\widetilde \lambda_i \left(g_i(x^t) \!+\! \langle\nabla g_i(x^t), x\!-\!x^t\rangle  \right)  \!-\! \frac{\widetilde L_{fg}}{2}\|x \!-\! \widetilde x\|^2,
\end{aligned}
\end{equation}
where the first equality makes use of \eqref{lamg0}.
On the other hand, since $f$ has Lipschitz continuous gradient (with modulus $L_f$), if $g(\widetilde x)\le 0$, then we have for any $x\in\R^n$ that
\begin{eqnarray*}
  &&F(\widetilde x)=f(\widetilde x)+ P_1(\widetilde x) - P_2(\widetilde x)\\
  &&\le f(x^t) +\left<\nabla f(x^t),\widetilde x - x^t\right> + \frac{L_f}{2}\|\widetilde x -x^t\|^2 + P_1(\widetilde x) - P_2(\widetilde x)\\
 &&= f(x^t) \!+\! \left<\nabla f(x^t),\widetilde x \!-\!x^t\right> \!+\! \frac{\widetilde L_f}{2}\|\widetilde x \!-\!x^t\|^2 \!+\! P_1(\widetilde x) \!-\! P_2(\widetilde x) \!-\! \frac{\widetilde L_f \!-\! L_f}{2}\|\widetilde x \!-\!x^t\|^2\\
    && \stackrel{{\rm (a)}}{\leq} f(x^t) + \left<\nabla f(x^t),\widetilde x -x^t\right> + \frac{\widetilde L_f}{2}\|\widetilde x -x^t\|^2 + P_1(\widetilde x)\\
    && \ \ \ \ - P_2(x^{t}) - \left<\xi^t,\widetilde x - x^t\right>- \frac{\widetilde L_f - L_f}{2}\|\widetilde x -x^t\|^2\\
    &&\leq f(x^t) + \left<\nabla f(x^t),x -x^t\right> + \frac{\widetilde L_{fg}}{2}\|x -x^t\|^2 + P_1(x) - P_2(x^{t}) - \left<\xi^t,x - x^t\right> \\
    &&\ \ \ \ +\sum_{i=1}^m\widetilde \lambda_i \left(g_i(x^t) + \langle\nabla g_i(x^t), x-x^t\rangle  \right)- \frac{\widetilde L_{fg}}{2}\|x - \widetilde x\|^2 - \frac{\widetilde L_f - L_f}{2}\|\widetilde x -x^t\|^2,
  \end{eqnarray*}
where (a) uses the convexity of $P_2$ and the fact that $\xi^t\in \partial P_2(x^t)$, while the last inequality holds due to \eqref{strongcpre}. This shows that \eqref{strongc} holds at the $t^{\rm th}$ iteration.

Now we show that there exists $k_0$ (independent of $t$) so that the inner loop in Step 3 terminates after finitely many iterations at the $t^{\rm th}$ iteration and returns an $x^{t+1}$ satisfying $g(x^{t+1})\le 0$.
To this end, let $k_1\in\N_+$ be such that $\b{L} \tau^{k_1} > \max\{\frac12(c + L_f), L_{g_1}, \ldots, L_{g_m}\}$. Then $k_1$ does not depend on $t$ and we have
 \begin{equation}\label{update}
  L_f^{t,0}\tau^{k_1} - \frac{ L_f}{2}\ge \b{L}\tau^{k_1} - \frac{ L_f}{2}>\frac{c}{2}{\rm \ and \ }
 (L_g^{t,0})_i\tau^{k_1} \ge \b{L}\tau^{k_1} \ge L_{g_i} \ {\rm for}\ i = 1,\ldots,m.
 \end{equation}
Note that for each $i$, since $g_i$ has Lipschitz gradient with Lipschitz modulus $L_{g_i}$, we have for any $(\widetilde L_{g})_i > 0$ that
 \begin{equation*}
 \begin{split}
   &g_i(\widetilde x)\le g_i(x^t) + \langle\nabla g_i(x^t),\widetilde x - x^t\rangle + \frac{L_{g_i}}{2}\|\widetilde x - x^t\|^2\\
   & = \bar G(\widetilde x,x^t,\widetilde L_g) +  \frac{L_{g_i} - (\widetilde L_{g})_i}{2}\|\widetilde x - x^t\|^2.
 \end{split}
 \end{equation*}
This together with \eqref{update} and the update rule of $\widetilde L_g$ in Step 3b) shows that after at most $k_1$ calls of Step 3b), we have $g(\widetilde x)\le 0$.
Whenever $\widetilde x$ satisfies $g(\widetilde x)\le 0$, we can apply \eqref{strongc} with $x$ being $x^t$ to conclude that
 \begin{equation*}
 \begin{split}
 &F(\widetilde x)\!\le\! f(x^t)\!+\! P_1(x^t) \!-\! P_2(x^t)\!+\!\langle\widetilde \lambda,g(x^t) \rangle\!-\! \frac{\widetilde L_{fg}}{2}\|x^t \!-\! \widetilde x\|^2 \!-\!  \bigg[\frac{\widetilde L_f - L_f}{2}\bigg]\|x^t -\widetilde x\|^2 \\
 &\le f(x^t)+ P_1(x^t) - P_2(x^t) - \frac{\langle\widetilde\lambda, \widetilde L_g\rangle}{2}\|x^t - \widetilde x\|^2-  \left[ \widetilde L_f -\frac{L_f}{2}\right]\|x^t -\widetilde x\|^2 \\
 &\le F(x^t) - \left[ \widetilde L_f -\frac{L_f}{2}\right]\|x^t -\widetilde x\|^2,
 \end{split}
 \end{equation*}
where the second inequality holds because $\widetilde\lambda\in\R_+^n$ and $g(x^t)\le 0$; we also used the fact that $\widetilde L_{fg} = \widetilde L_f + \langle\widetilde\lambda,\widetilde L_g\rangle$. Thus, in view of the above two displays, the conditions in Step 3a) must hold when $(\widetilde L_g)_i \ge L_{g_i}$ for all $i$ and \rerevise{$\widetilde L_f \ge \frac{L_f+c}2$}; according to the update rules of $\widetilde L_f$ and $\widetilde L_g$, this happens after at most $k_1$ calls of Step 3b) and $k_1$ calls of Step 3c). Thus, at iteration $t$, the inner loop stops after at most $k_0 := 2k_1$ iterations and outputs an $x^{t+1}$ satisfying $g(x^{t+1})\le 0$ \rerevise{and $F(x^{t+1})\le F(x^t) - \frac{c}2\|x^{t+1}-x^t\|^2$.}

Finally, since $g(x^0)\le 0$ to start with,
by induction, we know that for any $t\ge 0$, the inner loop stops after at most $k_0$ iterations. This together with the fact that $\{(L_f^{t,0},L_{g}^{t,0})\}\subseteq [\b{L},{\rm \bar L}]^{m+1}$ implies that $\{(L_f^t, L_{g}^t)\}$ is bounded.
Therefore, SCP$_{ls}$ is well defined and items (ii), (iii) and (iv) hold. This completes the proof.
\end{proof}

\section{Convergence properties of SCP$_{ls}$}\label{sec3}
\subsection{Convergence analysis in nonconvex settings}
In this section, we analyze SCP$_{ls}$ when $F$ in \eqref{Or} is possibly nonconvex.
We first prove some basic properties of the sequence generated by SCP$_{ls}$. Item (iii) in the following theorem was already proved in \cite[Theorem~3.7]{Lu10}; we also include its proof here for the ease of the readers.
\begin{theorem}\label{bastation}
Consider \eqref{Or} and suppose that Assumptions~\ref{basebase} and \ref{baseass} hold. Let $\{(x^t,L_{g}^t)\}$ be generated by SCP$_{ls}$. Then the following statements  hold:
  \begin{enumerate}[{\rm(i)}]
    \item The sequence $\{x^t\}$ is bounded.
    \item The sequence $\{\bar F(x^{t+1},x^{t},L_{g}^{t})\}$ is nonincreasing and convergent to some real number $\bar F^*$, where $\bar F$ is defined as in \eqref{barF}. Moreover, for any $t\ge 1$, we have
        \begin{align}\label{boundedblow}
        \bar F(x^{t+1},x^{t},L_g^t) \le \bar F(x^{t},x^{t - 1},L_{g}^{t-1}) - \frac{c}{2}\|x^{t+1} - x^t\|^2.
        \end{align}
    \item It holds that $\lim\limits_{t\to\infty}\|x^{t+1} - x^t\| = 0$.
  \end{enumerate}
\end{theorem}
\begin{proof}
 Let $F$ be defined as in \eqref{Or}.  Then for any $t\ge0$, we have
 \begin{align}\label{levelbound}
 F(x^{t+1}) - F(x^0)=  \sum_{i=0}^t[F(x^{i+1}) - F(x^i)]\le - \sum_{i=0}^t\frac{c}{2}\|x^{i+1} - x^i\|^2\le 0,
 \end{align}
 where the first inequality follows from \eqref{decrease}.  Since $F$ is level-bounded by Assumption~\ref{basebase}(iii), we deduce that $\{x^t\}$ is bounded and the conclusion in item (i) holds.

 We now prove (ii). Since for any $t\ge0$, the $x^{t+1}$ belongs to ${\rm dom}\,F$ and is feasible for \eqref{subp} with $(\widetilde L_f,\widetilde L_g) = (L^t_f,L^t_g)$, it holds that
 \begin{align}\label{eeeq}
 \bar F(x^{t+1},x^{t},L_g^t)= F(x^{t+1})\ \ {\rm for \ }t\ge 0.
 \end{align}
 This together with \eqref{decrease} shows that
$\{\bar F(x^{t+1},x^t,L_{g}^t)\}$ is nonincreasing and \eqref{boundedblow} holds for all $t\ge 1$. Also, thanks to \eqref{eeeq} and Assumption \ref{basebase}, we have
  \[
  \inf_t \bar F(x^{t+1},x^{t},L_g^t) = \inf_t F(x^t)\ge \inf F>- \infty,
  \]
  implying  that $\{\bar F(x^{t+1},x^{t},L_{g}^{t})\}$ is   bounded from below. Thus, we conclude that the sequence $\{\bar F(x^{t+1},x^{t},L_{g}^{t})\}$ is convergent. We denote this limit by $\bar F^*$.

Finally, we prove (iii). Since $\{\bar  F(x^{t+1},x^{t},L_{g}^{t})\}$ converges to $\bar F^*$, passing to the limit as $t$ goes to infinity in \eqref{levelbound} and invoking \eqref{eeeq}, we have
\[
 \sum_{i=0}^{\infty}\frac{c}{2}\|x^{i+1} - x^i\|^2\le F(x^0) -  \lim_{t\to\infty} \bar F(x^{t+1},x^t,L_g^t)  = F(x^0) -  \bar F^*<\infty.
\]
Therefore, item (iii) holds. This completes the proof.
\end{proof}

Next, we show that $\{\lambda^t\}$ with each $\lambda^t$ being a Lagrange multiplier\footnote{The existence of $\lambda^t$ follows from Lemma~\ref{rdayu0}(iv).} of \eqref{subp} with $(\widetilde L_f,\widetilde L_g) = (L^t_f,L^t_g)$ is bounded and
any cluster point of the sequence $\{x^t\}$ generated by SCP$_{ls}$ is a stationary point of \eqref{Or} in the sense of Definition~\ref{def:stat}. The latter conclusion was also proved in \cite[Theorem~3.7]{Lu10}. We include its proof for completeness.
\begin{theorem}\label{boundedlambda}
Consider \eqref{Or} and suppose that Assumptions~\ref{basebase} and \ref{baseass} hold.  Let $\{x^t\}$  be the  sequence generated by SCP$_{ls}$ and $\lambda^t$ be a Lagrange multiplier of \eqref{subp} with $(\widetilde L_f,\widetilde L_g) = (L^t_f,L^t_g)$. Then
  the sequence $\{\lambda^t\}$ is bounded and any accumulation point of $\{x^t\}$ is a stationary point of \eqref{Or}.
\end{theorem}
\begin{proof}
      Suppose to the contrary that $\{\lambda^t\}$ is unbounded and let $\{\lambda^{t_j}\}$ be a subsequence of $\{\lambda^t\}$ such that $\|\lambda^{t_j}\|\stackrel{j}{\to}\infty$. Passing to a further subsequence if necessary, we may assume that there exist $\lambda^*\in \R^m_+$ and $x^*$ such that  $\lim\limits_{j\to\infty}\frac{\lambda^{t_j}}{\|\lambda^{t_j}\|} = \lambda^*$ and $\lim\limits_{j\to\infty}x^{t_j}=x^*$, where the existence of $x^*$ is due to Theorem~\ref{bastation}(i).

       Using \eqref{firstoderofsubp}, the definition of $\widetilde L_{fg}$ there and the fact $(\widetilde L_f,\widetilde L_g) = (L^t_f,L^t_g)$, we have
   \begin{align*}
   \eta^t\!:=\! \sum_{i=1}^m\lambda_i^t\left[\nabla g_i(x^t)\! +\! (L_{g}^t)_i(x^{t+1} - x^t)\right]\! \in \!- \nabla f(x^t) - L_f^t (x^{t+1} - x^t)\! - \partial P_1(x^{t+1}) + \xi^t.
   \end{align*}
 Since the functions $\nabla f$, $P_1$ and $P_2$ are continuous, and  $\{(x^t,L_f^t)\}$ is bounded thanks to Theorem~\ref{bastation}(i) and Lemma~\ref{rdayu0}(ii), we deduce from the above display that $\{\eta^t\}$ is bounded. Then, dividing $\eta^{t_j}$  by $\|\lambda^{t_j}\|$ and letting $j\to \infty$, using the continuity of $\nabla g$ and  Theorem~\ref{bastation}(iii) together with Lemma~\ref{rdayu0}(ii), we deduce further that
   \begin{align}\label{mfcqcontr}
   \sum_{i=1}^m\lambda^*_i\nabla g_i(x^*) =0.
   \end{align}
   On the other hand, using \eqref{lamg0} with $(\widetilde x,\widetilde \lambda,\widetilde L_g)= (x^{t+1},\lambda^t,L_g^t)$,  the continuity of $\nabla g_i$ for each $i$, Lemma~\ref{rdayu0}(ii) and Theorem~\ref{bastation}(iii), we see that $\lambda_i^*g_i(x^*) = 0$ for all $i=1,\dots,m$.
   This further implies that
   \[
   \lambda^*_i=0 {\rm\ for\ } i\not\in I(x^*).
   \]
The above display and \eqref{mfcqcontr} imply that
\[
\sum_{i\in I(x^*)}\lambda^*_i\nabla g_i(x^*) =0.
\]
Combining this with MFCQ (Assumption~\ref{baseass}) and recalling that $\lambda^*\in \R^m_+$, we conclude that $\lambda^*_i = 0$ for $i\in I( x^*)$. Therefore, we have $\lambda^*=0$, contradicting the fact that $\|\lambda^*\|=1$. Thus, the sequence $\{\lambda^t\}$ is bounded.

For the second conclusion of this theorem, let $\bar x$ be an accumulation point of $\{x^t\}$ with $\lim\limits_{k\to\infty} x^{t_k} = \bar x$. Since $\{\lambda^t\}$ is bounded, passing to a further subsequence if necessary, we assume without loss of generality that $\lim\limits_{k\to\infty}\lambda^{t_k}=\bar \lambda$ for some $\bar\lambda$.
Since  the sequence $\{(L_f^{t},L_g^t,\lambda^t)\}$ is bounded thanks to  Lemma~\ref{rdayu0}(ii) and the boundedness of $\{\lambda^t\}$, using Theorem~\ref{bastation}(iii), we have that $\lim\limits_{k\to\infty}\left(L_f^{t_k} + \langle\lambda^{t_k},L_g^{t_k}\rangle\right)(x^{t_k+1} - x^{t_k}) = 0$. Using this fact together with the  closedness  of $\partial P_1$ and $\partial P_2$, the Lipschitz continuity of $\nabla f$ and $\nabla g$ and Theorem \ref{bastation}(iii), we have
upon passing to the limit as $k$ goes to infinity in \eqref{firstoderofsubp} with $(\widetilde x,\widetilde \lambda,\widetilde L_f,\widetilde L_g) = (x^{t_k+1},\lambda^{t_k},L_f^{t_k},L_g^{t_k})$ and $t = t_k$ that
\begin{equation}\label{station}
   \begin{split}
  0&\in \nabla f(\bar x)  + \partial P_1(\bar x) - \partial P_2(\bar x) + \sum_{i=1}^m\bar \lambda_i\nabla g_i(\bar x).
  \end{split}
  \end{equation}

  On the other hand, using \eqref{lamg0} with $(\widetilde x,\widetilde \lambda,\widetilde L_g)=(x^{t_k+1}, \lambda^{t_k},L_g^{t_k})$ and $t = t_k$, letting $k\to \infty$, we have upon using the continuity of $\nabla g$, Theorem~\ref{bastation}(iii) and Lemma~\ref{rdayu0}(ii) that
  \begin{align}\label{compl}
  \bar \lambda_i g_i(\bar x)=0 {\rm \ for\ all\ }i=1,\dots,m.
  \end{align}
   Finally, since $\lambda^t\ge 0$ for any $t\ge 0$, we have $\bar \lambda\ge 0 $. Also, since $g_i$ is continuous for each $i$ and $g(x^t)\le 0$ thanks to Step 3a) of SCP$_{ls}$, we have $g(\bar x)\le 0$. These together with \eqref{station} and \eqref{compl} imply that $\bar x$ is a stationary point of \eqref{Or}.
\end{proof}

\begin{lemma}\label{Lemma:barF}
  Consider \eqref{Or} and suppose that Assumptions~\ref{basebase} and \ref{baseass} hold. Let $\{(x^t,L_g^t)\}$ be the sequence generated by SCP$_{ls}$ and let $\Omega$ be the set of accumulation points of the sequence $\{(x^{t+1},x^{t},L_{g}^{t})\}$. Then $\Omega\neq \emptyset$ and $\bar F\equiv\bar F^*$ on $\Omega$, where $\bar F$ is defined as in \eqref{barF} and $\bar F^*$ is given in Theorem~\ref{bastation}(ii).
\end{lemma}
\begin{proof}
From Theorem~\ref{bastation}(i) and Lemma~\ref{rdayu0}(ii) we know that $\Omega\neq\emptyset$.
Fix any $(x^{\Omega},y^{\Omega},L^{\Omega})\in\Omega$ and let $\{(x^{t_j+1},x^{t_j},L_{g}^{t_j})\}$ be a convergent subsequence with $\lim\limits_{j\to\infty}(x^{t_j+1},x^{t_j},L_{g}^{t_j})=(x^{\Omega},y^{\Omega},L^{\Omega})$.
  Since each $\nabla g_i$ is continuous and $x^{t_j+1}$ belongs to ${\rm dom}\,F$ and is feasible for \eqref{subp} with $t = t_j$ and $(\widetilde L_f,\widetilde L_g) = (L_f^{t_j},L_g^{t_j})$, we have
   \begin{equation}\label{barGzero}
   g(x^\Omega) = \lim_{j\to\infty}g(x^{t_j+1})\le 0,\ \ \ \bar G(x^{\Omega},y^{\Omega},L^{\Omega}) = \lim_{j\to\infty} \bar G(x^{t_j+1},x^{t_j},L_{g}^{t_j})\le 0
   \end{equation}
   and $F(x^{t_j+1}) = \bar F(x^{t_j+1},x^{t_j},L_{g}^{t_j})$ for all $j$. Then, using the continuity of $F$ on its closed domain, we have
  \[
  F(x^\Omega) = \lim_{j\to\infty} F(x^{t_j+1}) = \lim_{j\to \infty}\bar F(x^{t_j+1},x^{t_j},L_{g}^{t_j}) = \bar F^*,
  \]
  where the last equality follows from Theorem~\ref{bastation}(ii).
  Thus, we deduce that
  \[
  \bar F(x^{\Omega},y^{\Omega},L^{\Omega}) = F(x^\Omega) =  \bar F^*,
  \]
  where the first equality follows from \eqref{barGzero}.
  Since $(x^{\Omega},y^{\Omega},L^{\Omega})\in\Omega$ is arbitrary, we conclude that $\bar F \equiv \bar F^*$ on $\Omega$.
\end{proof}

To analyze the global convergence properties of SCP$_{ls}$, we need a bound on the subdifferential of $\bar F$ in \eqref{barF}. To this end, we consider the following additional differentiability assumption on $g_i$.
\begin{assumption}\label{diffofgi}
Each $g_i$ in \eqref{Or} is twice continuously differentiable.
\end{assumption}
\begin{lemma}\label{subofbarF}
Consider \eqref{Or} and suppose that Assumption \ref{diffofgi} holds. Let $(x,y,w)\in \R^n\times \R^n\times \R^{m}$ and assume that $P_2$ is continuously differentiable around $x$. Then
  \begin{equation}\label{partial}
    \partial \bar F(x,y,w)\supseteq \begin{pmatrix}
      \nabla  f(x)  - \nabla P_2(x) + \partial P_1(x) + \sum_{i=1}^m\lambda_i[\nabla g_i(y) + w_i(x - y)]\\
    \sum_{i=1}^m\lambda_i[\nabla^2 g_i(y)(x-y) - w_i(x - y)]\\
    \frac12\|x-y\|^2\lambda
   \end{pmatrix}
    \end{equation}
whenever $\lambda \in N_{-\R^m_+}(\bar G(x,y,w))$, where $\bar F$ and $\bar G$ are defined as in \eqref{barF}.
\end{lemma}
\begin{proof}
   We only consider the case where $(x,y,w)\in {\rm dom} \bar F$, since \eqref{partial} holds trivially otherwise.
   Using \cite[Exercise~8.8, Corollary~10.9, Proposition~10.5]{RoWe97}, we have
  \begin{align*}
    &\partial \bar F(x,y,w)\supseteq \hat\partial \bar F(x,y,w) \supseteq \begin{pmatrix}
      \nabla  f(x)  - \nabla P_2(x)  + \hat\partial P_1(x)\\
    0\\
    0
   \end{pmatrix} + \hat \partial \delta_{\bar G(\cdot)\leq 0}(x,y,w)\\
    &\overset{\rm (a)}= \begin{pmatrix}
      \nabla  f(x)  - \nabla P_2(x)   +\partial P_1(x)\\
    0\\
    0
   \end{pmatrix} + \hat N_{\bar G(\cdot)\leq 0}(x,y,w)\\
    &\overset{\rm (b)}\supseteq  \begin{pmatrix}\nabla  f(x)  - \nabla P_2(x) + \partial P_1(x)\\
    0\\
    0
   \end{pmatrix} + \sum_{i=1}^m \lambda_i \begin{pmatrix}
    \nabla g_i(y) + w_i(x - y) \\
    \nabla^2 g_i(y)(x-y) - w_i(x - y) \\
    \frac12\|x-y\|^2e_i
    \end{pmatrix},
    \end{align*}
where (a) uses the convexity of $P_1$ and \cite[Proposition~8.12]{RoWe97}, $e_i\in \R^m$ is the $i^{\rm th}$ standard basis vector and (b) holds for any $ \lambda\in \hat N_{-\R^m_+}(\bar G(x,y,w)) = N_{-\R^m_+}(\bar G(x,y,w))$, thanks to \cite[Theorem~6.14]{RoWe97}.
\end{proof}

We also need the following assumption to derive the desired bound on $\partial \bar F$. This assumption was also used in \cite{WeChPo18} for analyzing the global convergence property of the sequence generated by the \revise{proximal DCA with extrapolation (pDCA$_e$).}
\begin{assumption}\label{p1p2lip}
Each $g_i$ in \eqref{Or} is smooth, and the $P_2$ in \eqref{Or} is  continuously differentiable on an open set $\Gamma$ that contains all stationary points of \eqref{Or}. Moreover, the function  $\nabla P_2$ is  locally Lipschitz continuous on $\Gamma$.
\end{assumption}
Using this assumption and Lemma~\ref{subofbarF}, we can prove the following property of $\partial\bar F$.
\begin{lemma}
 Consider \eqref{Or} and suppose that Assumptions~\ref{basebase}, \ref{baseass},  \ref{diffofgi} and  \ref{p1p2lip}  hold.  Let $\{(x^t,L_g^t)\}$ be the  sequence generated by SCP$_{ls}$ and let $\bar F$ be defined as in \eqref{barF}. Then there exist $\kappa>0$ and $\underline{t}\in\mathbb{N}_+$ such that
  \begin{align}\label{eq:bdsubdiff}
    {\rm dist}(0, \partial\bar F(x^{t+1},x^t,L_{g}^t)) \leq \kappa\|x^{t+1}- x^t\|\ {\rm \ for \ all \ }t>\underline{t}.
  \end{align}
\end{lemma}
\begin{proof}
From Theorem~\ref{bastation}(i), we know that $\{x^t\}$ is bounded. Thus, denoting the set of accumulation points of $\{x^t\}$ as $\Omega_x$, we have that  $\Omega_x$ is compact and $\Omega_x\subseteq \Gamma$ thanks to Theorem~\ref{boundedlambda}, where $\Gamma$ is the open set give in Assumption~\ref{p1p2lip}. Choose an $\epsilon>0$ so that $\Gamma_\epsilon:= \{x:\;{\rm dist}(x,\Omega_x)<\epsilon\}\subseteq \Gamma$ and $\nabla P_2$ is Lipschitz continuous with modulus $L_{P_2}$ on $\Gamma_\epsilon$, which exists thanks to the compactness of $\Omega_x$ and Assumption~\ref{p1p2lip}. Moreover, since $\Omega_x$ is compact, from the definition of cluster points, we see that there exists $t_0\in\mathbb{N}_+$ such that ${\rm dist}(x^t,\Omega_x)<\epsilon$ whenever $t > t_0$. In particular, $P_2$ is continuously differentiable around each $x^t$ whenever $t > t_0$. In addition, thanks to Theorem~\ref{bastation}(iii), we can further choose $\underline{t} > t_0 + 1$ such that for $t>\underline{t}$, we have
\begin{align}\label{lastlast}
\|x^{t+1}- x^t\|^2\leq \|x^{t+1}- x^t\|.
\end{align}

Now, let $\lambda^t$ be a Lagrange multiplier of \eqref{subp} with $(\widetilde L_f,\widetilde L_g) = (L^t_f,L^t_g)$, which exists thanks to Lemma~\ref{rdayu0}(iv). Then it holds that $\lambda^t\in N_{-\R^m_+}(\bar G(x^{t+1},x^t,L_{g}^t))$. Therefore, using  \eqref{partial}  with  $ \lambda = \lambda^t$ for any $t > \underline{t}$, we have that
\begin{align}\label{distdist}
 \partial \bar F(x^{t+1},x^t,L_{g}^t)\supseteq  \left(
    \begin{array}{c}
     J^t\\
      \sum_{i=1}^m \lambda_i^t\left(\nabla^2 g_i(x^t)(x^{t+1}-x^t) - (L_{g}^t)_i(x^{t+1} - x^t)\right)\\
     \frac12\|x^{t+1} - x^t\|^2\lambda^t
    \end{array}
    \right)
\end{align}
with
$
J^t:=\nabla  f(x^{t+1}) + \partial P_1(x^{t+1}) - \nabla P_2(x^{t+1})
+  \sum_{i=1}^m\lambda_i^t\left(\nabla g_i(x^t) + (L_{g}^t)_i(x^{t+1} - x^t) \right)
$.
For this $J^t$, using \eqref{firstoderofsubp} with $\widetilde x = x^{t+1}$ and recalling the definition of $\xi^t$, we have that
\begin{equation*}
   \begin{split}
     J^t\ni&\nabla  f(x^{t+1}) - \nabla P_2(x^{t+1})
+  \sum_{i=1}^m\lambda_i^t\left(\nabla g_i(x^t) + (L_{g}^t)_i(x^{t+1} - x^t)\right)\\
&\!\!+\left(-\nabla f(x^t) - L^t_f(x^{t+1} - x^t) + \nabla P_2(x^t) - \sum_{i=1}^m\lambda_i^t\left(\nabla g_i(x^t)  + (L_{g}^t)_i(x^{t+1} - x^t)\right)\right)\\
=&\nabla  f(x^{t+1}) -\nabla f(x^t) +\nabla P_2(x^t) - \nabla P_2(x^{t+1}) - L^t_f(x^{t+1} - x^t).
   \end{split}
 \end{equation*}
 Using this together with Cauchy-Schwarz inequality, for $t>\underline{t}$, it holds that
\begin{equation}\label{distdist2}
\begin{split}
 & \|J^t\|^2 \!\!\leq\! 3\bigg(\!\!\|\nabla f(x^{t+1}) \!-\!\nabla f(x^t)\|^2 \!\!+\! \|\nabla P_2(x^{t+1}) \!-\! \nabla P_2(x^t)\|^2  \!\!+\! \| L^t_f(x^{t+1} \!-\! x^{t})\|^2\!\!\bigg)\\
   &\stackrel{\rm (a)}{\leq} 3L_f^2\|x^{t+1}- x^t\|^2 + 3L_{P_2}^2\|x^{t+1}- x^t\|^2 + 3(L^t_f)^2\| x^{t+1} - x^{t}\|^2 \\
    & =\bigg(3L_f^2+ 3(L^t_f)^2+3L_{P_2}^2\bigg) \|x^{t+1}- x^t\|^2,
\end{split}
\end{equation}
where (a) makes use of the fact that $t > \underline{t}$ (so that $x^t\in \Gamma_\epsilon$) and the Lipschitz continuity of $\nabla f$ and $\nabla P_2$.

On the other hand, since $\{(x^t,L_{g}^t,\lambda^t)\}$ is bounded thanks to Theorem~\ref{bastation}(i), Lemma~\ref{rdayu0}(ii) and Theorem~\ref{boundedlambda}, using the continuity of $\nabla^2g_i$ for each  $i$, there exists $D_1>0$  such that
\begin{equation}\label{D1}
\begin{split}
 &\bigg\|\sum_{i=1}^m \lambda_i^t\bigg(\nabla^2 g_i(x^t)(x^{t+1}-x^t) - (L_{g}^t)_i(x^{t+1} - x^t)\bigg)\bigg\|^2\\
  &\le\! m\!\sum_{i=1}^m({\lambda_i^t})^2 \|\nabla^2 g_i(x^t)(x^{t+1}-x^t) - (L_{g}^t)_i(x^{t+1} - x^t)\|^2\!\leq\! D_1\|x^{t+1}-x^t\|^2,
 \end{split}
\end{equation}
where the first inequality uses the Cauchy-Schwarz inequality.

Therefore, since $\{(L^t_f,\lambda^t)\}$ is bounded thanks to Lemma~\ref{rdayu0}(ii) and Theorem~\ref{boundedlambda}, combining \eqref{lastlast}, \eqref{distdist}, \eqref{distdist2} and \eqref{D1}, we conclude that there exists $\kappa>0$ such that \eqref{eq:bdsubdiff} holds. This completes the proof.
\end{proof}

Now, if we suppose in addition that  $\bar F$ is a KL function with exponent $\alpha\in [0,1)$, then using the results above and following the analysis in \cite{AtBo09,AtBoReSo10,AtBoSv13,BoSaTe14,LiChPo19,WeChPo18}, we can deduce the convergence of the sequence $\{x^t\}$ generated by SCP$_{ls}$ to a stationary point of \eqref{Or} and estimate its local convergence rate. Specifically, using similar proofs as in \cite{LiChPo19,WeChPo18}, we have the following results. The lines of arguments are standard and we omit its proof for brevity.

\begin{theorem}[\textbf{Convergence rate of SCP$_{ls}$ in nonconvex settings}]\label{bastationKL}
Consider \eqref{Or}.
  Suppose that Assumptions~\ref{basebase}, \ref{baseass},  \ref{diffofgi} and \ref{p1p2lip} hold, and $\bar F$ in \eqref{barF} is \rerevise{a KL function.}
   \rerevise{Let $\{(x^t,L_g^t)\}$ be the sequence generated by SCP$_{ls}$ and let $\Omega$ be the set of accumulation points of the sequence $\{(x^{t+1},x^{t},L_{g}^{t})\}$.} Then $\{x^t\}$ converges to a stationary point $x^*$ of \eqref{Or}. Moreover, \rerevise{if $\bar F$ satisfies the KL property with exponent $\alpha\in[0,1)$ at every point in $\Omega$, then} there exists $\underline{t}\in\mathbb{N}_+$ such that the following statements hold:
  \begin{enumerate}[{\rm (i)}]
    \item If $\alpha=0$, then $\{x^t\}$ converges finitely, i.e., $x^t \equiv x^*$ for $t>\underline{t}$.
    \item If $\alpha\in(0,\frac12]$, then there exist $a_0\in(0,1)$ and $a_1>0$ such that
        \[
         \|x^t - x^*\|\le a_1a_0^{t}\  {\rm \ for \ }t>\underline{t}. 
          \]
    \item If $\alpha\in(\frac12,1)$, then there exists $a_2>0$ such that
        \[
        \|x^t - x^*\|\le a_2t^{-\frac{1-\alpha}{2\alpha - 1}}\ {\rm \ for \ }t>\underline{t}. 
        \]
  \end{enumerate}
\end{theorem}

\subsection{Convergence analysis in convex settings}
In this section, we study the convergence properties of SCP$_{ls}$ under the following convex settings:
\begin{assumption}\label{convexassumption}
  Suppose that in \eqref{Or}, $P_2= 0$ and  $\{f,g_1,\dots,g_m\}$ are convex.
\end{assumption}

\revise{Assumption~\ref{convexassumption} was also considered in \cite[Section~3.2.3]{BoChPa19} for analyzing MBA, and in \cite[Section~4]{BoChPa19} for its line search variant Multiprox$_{bt}$ \cite[Eq.~(37)]{BoChPa19}. Here, we would like to point out that the line search criterion in Multiprox$_{bt}$ \cite[Eq.~(37)]{BoChPa19} is different from the criterion \eqref{decrease} used in SCP$_{ls}$. The criterion in Multiprox$_{bt}$ relies on a local majorant of the objective function, while \eqref{decrease} uses the objective function directly, and is originated from SpaRSA; see \cite[Eq.~(22)]{WrNo09}. We will establish global convergence of the whole sequence generated by SCP$_{ls}$ in the above convex settings, under suitable assumptions.}
  Unlike the analysis in the previous subsection, our analysis here is based on KL property of $F$ in \eqref{Or} instead of that of $\bar F$, and we will {\em not} assume $g$ to be twice continuously differentiable (i.e., we do not require Assumption \ref{diffofgi}).
We start with two auxiliary lemmas. The first lemma is an analogue of \cite[Lemma~6]{BoSaTe14} and follows immediately from an application of \cite[Theorem~5]{BoNgPeSu17} and standard compactness argument. We omit the proof for brevity.

\begin{lemma}\label{uniformError}
  Let $f:\R^n\to (-\infty,+\infty]$ be a level-bounded proper closed convex function with $\Lambda := \Argmin f \neq\emptyset$. Let $\underline f := \inf f$.  Suppose that $f$ satisfies the KL property at each point in $\Lambda$ with exponent $\alpha\in[0,1)$. Then there exist $\epsilon>0$, $r_0>0$ and $c_0>0$ such that
  \[
  {\rm dist}(x, \Lambda)\le c_0(f(x) - \underline f)^{1-\alpha}
  \]
  for any $x\in{\rm dom }\partial f$ satisfying ${\rm dist}(x,\Lambda)\le\epsilon$ and $\underline f\le f(x)<\underline f + r_0$.
\end{lemma}

The next lemma is an analogue of Lemma~\ref{Lemma:barF} for $F$ in \eqref{Or}.

\begin{lemma}\label{constantbarForF}
 Consider \eqref{Or} and suppose that Assumptions~\ref{basebase} and \ref{baseass} hold. Let $\{x^t\}$ be the sequence generated by SCP$_{ls}$ for \eqref{Or} and let $\Omega_x$ be the set of accumulation point of $\{x^t\}$. Then the following statements hold:
  \begin{enumerate}[{\rm (i)}]
    \item It holds that $\Omega_x\neq\emptyset$ and $F\equiv\bar F^*$ on $\Omega_x$, where $F$ is defined as in \eqref{Or} and $\bar F^*$ is given in Theorem~\ref{bastation}(ii).
    \item The sequence $\{F(x^t)\}$ is nonincreasing and convergent to $\bar F^*$.
  \end{enumerate}
\end{lemma}
\begin{proof}
We note first from Theorem~\ref{bastation}(i) that $\Omega_x\neq\emptyset$. In addition, since $x^t\in {\rm dom}\,F$ and is feasible for \eqref{subp} (with $(t-1,L_f^{t-1},L_g^{t-1})$ in place of $(t,\widetilde L_f,\widetilde L_g)$), we have
\begin{align}\label{FbarF}
  F(x^t)  = f(x^t) + P_1(x^t) - P_2(x^t)= \bar F(x^t,x^{t-1},L_g^{t-1}),\ {\rm for\ all\ }t\ge1.
\end{align}
Fix any $x^*\in \Omega_x$ and let $\lim\limits_{j\to\infty}x^{t_j}=x^*$. Using the continuity of $F$ on its closed domain and \eqref{FbarF}, we see that
  \[
  F(x^*) = \lim_{j\to\infty} f(x^{t_j}) +P_1(x^{t_j}) - P_2(x^{t_j})= \lim_{j\to\infty}\bar F(x^{t_j},x^{{t_j}-1},L_g^{{t_j}-1}) = \bar F^*,
  \]
  where the last equality makes use of Theorem~\ref{bastation}(ii). This proves (i).
  The conclusion in (ii) now follows immediately upon combining the above display and \eqref{FbarF} with Theorem~\ref{bastation}(ii). This completes the proof.
\end{proof}

Now we present our main result in this subsection.
\begin{theorem}[\textbf{Convergence rate of SCP$_{ls}$ in convex settings}]\label{convexKL}
  Consider \eqref{Or} and suppose that Assumptions~\ref{basebase}, \ref{baseass} and \ref{convexassumption} hold. Let $\{x^t\}$ be the sequence generated by SCP$_{ls}$. Then $\{x^t\}$ converges to a minimizer  $x^*$ of \eqref{Or}.
  If in addition $F$ in \eqref{Or} is a KL function with exponent $\alpha\in[0,1)$, then the following statements hold:
  \begin{enumerate}[{\rm(i)}]
    \item If $\alpha\in[0,\frac12]$, then there exist $c_0 >0$, $Q_1\in(0,1)$ and $\underline{t}\in\mathbb{N}_+$, such that
         \[
    \|x^t - x^*\|\le c_0\, Q_1^{t} {\rm \ \ for \ }t>\underline{t}.
  \]
    \item If $\alpha\in(\frac12,1)$, then there exist $c_0 >0$ and $\underline{t}\in \mathbb{N}_+$ such that
     \[
         \|x^t - x^*\|\!\le\! c_0\, t^{ -\frac{1 -\alpha}{2\alpha - 1}} {\rm \ for\ } t>\underline{t}.
        \]
  \end{enumerate}
\end{theorem}
\begin{proof}
Let $S:=\Argmin F $ for notational simplicity. Note that $S\neq \emptyset$ thanks to Assumption~\ref{basebase}.
Since $P_2=0$ and $\{f,g_1,\dots,g_m\}$ are convex by Assumption~\ref{convexassumption}, using Theorem~\ref{boundedlambda} and \cite[Theorem~28.3]{Ro70}, we see that
\begin{align}\label{omgexsubS}
\emptyset\neq \Omega_x\subseteq S,
\end{align}
where $\Omega_x$ is as in Lemma~\ref{constantbarForF}.
 This together with Lemma~\ref{constantbarForF} implies that $\bar F^* = \inf F$.

Next, let $\lambda^t$ be a Lagrange multiplier of \eqref{subp} with $(\widetilde L_f,\widetilde L_g) = (L^t_f,L^t_g)$, which exists thanks to Lemma~\ref{rdayu0}(iv).
Since $P_2= 0$ and $g(x^t)\le 0$ for all $t$, for any $\bar x\in S$, using \eqref{strongc} with $x = \bar x$, $\widetilde x = x^{t+1}$, $\widetilde \lambda = \lambda^t$, $\widetilde L_f = L_f^t$ and $\widetilde L_{fg} = L_{fg}^t: = L_f^t + \langle\lambda^t,L_{g}^t\rangle$, we deduce that
\begin{align*}
  &F(x^{t+1})\le  f(x^t) + \left<\nabla f(x^t),\bar x -x^t\right> + P_1(\bar x) + \frac{L_{fg}^t}{2}\|\bar x -x^t\|^2-\frac{L_{fg}^t}{2}\|\bar x - x^{t+1}\|^2   \\
     &\ \ \ \ \ \ \ \ \ \ \ \ \ \ +\sum_{i=1}^m\lambda^{t}_i \left(g_i(x^t) + \langle\nabla g_i(x^t), \bar x-x^t\rangle  \right) - \frac{ L_f^t - L_f}{2}\|x^{t+1}  - x^t\|^2\\
  &\stackrel{{\rm (a)}}{\le}  f(\bar x) + P_1(\bar x) + \frac{L_{fg}^t}{2}\|\bar x -x^t\|^2  - \frac{L_{fg}^t}{2}\|\bar x - x^{t+1}\|^2- \frac{L_f^t - L_f}{2}\|x^{t+1} -x^t\|^2\\
  &\stackrel{\rm (b)}{\leq}\! f(\bar x) + P_1(\bar x) + \frac{L_{fg}^{t}}{2}\|\bar x -x^{t}\|^2 - \frac{L_{fg}^{t}}{2}\|x^{t+1}\! - \bar x\|^2 + \frac{ (L_f - L^{t}_f)_+}{c}( F(x^{t}) - F(x^{t+1}))\\
  &\leq f(\bar x) + P_1(\bar x) + \frac{L_{fg}^{t}}{2}\|\bar x -x^{t}\|^2 - \frac{L_{fg}^{t}}{2}\|x^{t+1}\! - \bar x\|^2 + \frac{ M_0}{c}( F(x^{t}) - F(x^{t+1})),
  \end{align*}
 where (a) holds because $\{f,g_1,\dots,g_m\}$ are convex, and $\lambda^t_i\ge 0$ and $g_i(\bar x)\le 0$ for all $i$, (b) follows from \eqref{decrease}, and the $M_0$ in the last inequality is an upper bound of $\{(L_f - L_f^{t})_+\}$, which exists thanks to Lemma~\ref{rdayu0}(ii). Rearranging terms in the above inequality and noting $\bar F^* = \inf F = f(\bar x) + P_1(\bar x)$ whenever $\bar x\in S$, we have for any $\bar x\in S$ that
  \begin{align*}
    \frac{F(x^{t+1})-\bar F^*}{L_{fg}^t}  \le \frac{1}{2}\|\bar x -x^t\|^2 - \frac{1}{2}\|x^{t+1} - \bar x\|^2+\frac{M_0}{c L_{fg}^{t}}\left( F(x^{t}) - F(x^{t+1})\right).
  \end{align*}
  Let $L_{\max} $ be the upper bound of $\{L_{fg}^t\}$ (which exists according to Lemma~\ref{rdayu0}(ii) and Theorem~\ref{boundedlambda}) and recall that $L^t_{fg}\ge L^t_f \ge \b{L} > 0$ for all $t$, where $\b{L}$ is the one used in Step 2 of SCP$_{ls}$. Then we have from the above display that  for any $\bar x\in S$,
   \begin{align*}
    &\gamma\left(F(x^{t+1})-\bar F^*\right)\le \frac{1}{2}\|\bar x -x^{t}\|^2 - \frac{1}{2}\|x^{t+1} - \bar x\|^2+\theta\left( F(x^{t}) - F(x^{t+1})\right),
  \end{align*}
  where $\gamma := \frac1{L_{\max}}$ and $\theta := \frac{M_0}{c\b{L}}$.
  Rearranging terms in the above inequality, we have
    \begin{align}\label{convergence1}
\left(\gamma+\theta\right)\left(F(x^{t+1}) - \bar F^*\right) \le  \frac{1}{2}\|\bar x -x^t\|^2 -\frac{1}{2}\|x^{t+1} - \bar x\|^2  +\theta\left( F(x^t) - \bar F^*\right).
\end{align}
The inequality above in particular implies that for any $\bar x\in S$,
\begin{equation}\label{summable}
\begin{split}
&\frac{1}{2}\|x^{t+1} - \bar x\|^2 \le  \frac{1}{2}\|\bar x -x^{t}\|^2   +\theta\left( F(x^{t}) - \bar F^*\right) - \left(\gamma+\theta\right)\left(F(x^{t+1}) - \bar F^*\right)\\
&\le \frac{1}{2}\|\bar x -x^t\|^2   +(\gamma +\theta)\left( F(x^t) - F(x^{t+1})\right),
\end{split}
\end{equation}
where the last inequality holds because $\bar F^* = \inf F\le F(x^t)$.
Since $\{F(x^{t}) - F(x^{t+1})\}$ is nonnegative and summable thanks to Lemma~\ref{constantbarForF}(ii), using \eqref{omgexsubS}, \eqref{summable} and \cite[Proposition~1]{Iu03}, we conclude that $\{x^t\}$ converges to a minimizer $x^*$ of \eqref{Or}.

Now, we suppose in addition that $F$ is a KL function with exponent $\alpha\in [0,1)$. Let $\bar{x}^t\in S$ satisfy $\|x^t - \bar{x}^t\|={\rm dist}(x^t,S)$.
\revise{Since $\bar x^t\in S$, it holds  that $-\| x^{t+1} - \bar x^{t} \|^2 \le -{\rm dist}^2(x^{t+1},S)$.  Using this and applying \eqref{convergence1} with $\bar{x}^t$ in place of $\bar x$ gives}
\begin{align}\label{fbeta}
\left(\gamma+\theta\right)\left(F(x^{t+1}) - \bar F^*\right)&\le \frac{1}{2}{\rm dist}^2(x^t,S)   -\frac{1}{2}{\rm dist}^2(x^{t+1},S) +\theta\left( F(x^t) - \bar F^*\right).
\end{align}
For notational simplicity, let
\begin{align}\label{betak}
\beta_t := F(x^t) - \bar F^* + \frac{1}{2(\gamma + \theta)}{\rm dist}^2(x^t,S).
\end{align}
Using this, rearranging terms and dividing $\gamma+\theta$ from both sides of \eqref{fbeta}, we have
  \begin{equation}\label{convexss1}
   \beta_{t+1}\le    \frac{\theta}{\gamma + \theta}\left( F(x^t) - \bar F^*\right)+\frac{1}{2(\gamma + \theta)}{\rm dist}^2(x^t,S).
  \end{equation}

 Since $F$ is a proper closed convex level-bounded KL function with exponent $\alpha\in [0,1)$, using Lemma~\ref{uniformError}, there exist $0<\bar a<1$, $\bar c>0$ and $0<\epsilon<1$ such that
  \begin{align}\label{KL1}
    {\rm dist}(x,S)^{\frac{1}{1-\alpha}}\le\bar c\left( F(x) - \bar F^*\right)
  \end{align}
  for any $x\in{\rm dom}\partial F$ satisfying ${\rm dist}(x,S)\le \epsilon$ and $\bar F^*\le F(x)<\bar F^* +\bar a$.

Clearly, $\{x^t\} \subset {\rm dom}\partial F = \{x:\; g(x)\le 0\}$.
Next, since $\{x^t\}$ is bounded thanks to Theorem~\ref{bastation}(i), using \eqref{omgexsubS}, there exists $t_1$ such that
 \begin{align}\label{KL2}
 {\rm dist}(x^{t},S) \le{\rm dist }(x^{t},\Omega_x)<\epsilon,\ {\rm for\ }t>t_1.
 \end{align}
 On the other hand, using Lemma~\ref{constantbarForF}(ii), we see that there exists $t_2$ such that
 \begin{align}\label{KL3}
\bar F^*\le F(x^t) <\bar F^*+ \bar a, {\rm\ for \ }t>t_2.
 \end{align}

We now consider the cases when $\alpha\in [0,\frac12]$ and $\alpha\in (\frac12,1)$ separately.
\paragraph{\bf Case (i)} $\alpha \in [0,\frac12]$. Combining \eqref{KL1}, \eqref{KL2} and \eqref{KL3}, we conclude that for any $t>t_3:=\max\{t_1,t_2\}$,
  \begin{equation}\label{KL4}
  \begin{split}
 {\rm dist}^2(x^t,S) &\le {\rm dist}^{\frac{1}{1-\alpha}}(x^t,S)\le \bar c\left(F(x^t) - \bar F^*\right),
   \end{split}
   \end{equation}
    where the first inequality holds because $\frac{1}{1-\alpha}\le2$ and ${\rm dist}(x^t,S)<\epsilon<1$. Next,
    let $\zeta := \frac{2\theta + \bar c}{2(\gamma+\theta)+\bar c}\in (0,1)$.
Then one can show that
\begin{equation}\label{hahahaha}
\frac{\theta}{\gamma+\theta} + \frac{(1-\zeta)\bar c}{2(\gamma+\theta)} = \zeta.
\end{equation}
Using this and \eqref{convexss1}, we have for all $t > t_3$ that
\[
\begin{aligned}
\beta_{t+1} &\le \frac{\theta}{\gamma+\theta}(F(x^t)-\bar F^*) + \frac{1-\zeta}{2(\gamma+\theta)} {\rm dist}^2(x^t,S) + \frac{\zeta}{2(\gamma+\theta)} {\rm dist}^2(x^t,S) \\
& \overset{\rm (a)}\le \left(\frac{\theta}{\gamma+\theta} + \frac{(1-\zeta)\bar c}{2(\gamma+\theta)} \right)(F(x^t)-\bar F^*) + \frac{\zeta}{2(\gamma+\theta)} {\rm dist}^2(x^t,S) \\
& \overset{\rm (b)}= \zeta\left(F(x^t)-\bar F^* + \frac{1}{2(\gamma+\theta)} {\rm dist}^2(x^t,S)\right) = \zeta\beta_t,
\end{aligned}
\]
where (a) follows from \eqref{KL4} and (b) follows from \eqref{hahahaha}.
Combining the above inequality with the definition of $\beta_t$ in \eqref{betak} gives
\begin{align}\label{induction2}
 F(x^t) - \bar F^*\le \beta_t\le\zeta^{t - t_3 - 1}\beta_{t_3 + 1}\ {\rm \ for \ }t>t_3.
\end{align}
Then, for $t>t_3$, we have
  \begin{align*}
    &\|x^* - x^t\| \le \sum_{j=t+1}^{\infty}\|x^j - x^{j-1}\|\le\sum_{j=t+1}^{\infty}\sqrt{\frac{2}{c}}\sqrt{ F(x^{j-1}) - F(x^{j})}\\
    &\leq\sum_{j=t+1}^{\infty}\sqrt{\frac{2}{c}}\sqrt{ F(x^{j-1}) - \bar F^* }\le \sum_{j=t+1}^{\infty}\sqrt{\frac{2}{c}}\sqrt{\zeta^{j-t_3-2}\beta_{t_3+1}}
    = \sqrt{\frac{2\beta_{t_3+1}}{c\zeta^{t_3+1}}}\frac{(\sqrt{\zeta})^t}{1-\sqrt{\zeta}},
  \end{align*}
   where the second inequality follows from \eqref{decrease}, the third inequality follows from Lemma~\ref{constantbarForF}(ii) and the last inequality follows from \eqref{induction2}. This proves (i).

\paragraph{\bf Case (ii)} $\alpha\in (\frac12,1)$.
Using \eqref{convexss1} and the definition of $\beta_t$ in \eqref{betak}, for any $t>t_3 = \max\{t_1,t_2\}$,  we have
\begin{align*}
&\beta_{t+1} \le \beta_t - \frac{\gamma}{\gamma + \theta}\left(F(x^t) - \bar F^*\right)\\
&=\beta_t - \frac12c_3\left[F(x^t) - \bar F^* + \bar c\left(\frac{1}{2(\gamma + \theta)}\right)^{\frac{1}{2(1- \alpha)}}\left(F(x^t) - \bar F^*\right)\right]\\
&\overset{\rm (a)}\le\beta_t - \frac12c_3\left[F(x^t) - \bar F^* + \left(\frac{1}{2(\gamma + \theta)}\right)^{\frac{1}{2(1- \alpha)}}{\rm dist}(x^t,S)^{\frac{1}{1- \alpha}}\right]\\
&\overset{\rm (b)}\le \beta_t - \frac12c_3\left[\left(F(x^t) - \bar F^*\right)^{\frac{1}{2(1- \alpha)}} + \left(\frac{1}{2(\gamma + \theta)}{\rm dist}^2(x^t,S)\right)^{\frac{1}{2(1- \alpha)}}\right],
\end{align*}
where $c_3 =2\frac{\frac{\gamma}{\gamma + \theta}}{1+\bar c \left(\frac{1}{2(\gamma + \theta)}\right)^{\frac{1}{2(1- \alpha)}}}$, (a) follows from \eqref{KL1}, \eqref{KL2}, \eqref{KL3} and the fact that $\{x^t\}\subset {\rm dom}\partial F=\{x:\; g(x)\le 0\}$, and (b) holds because  $0\le F(x^t) - \bar F^*<\bar a<1$ (thanks to \eqref{KL3}) and $\frac{1}{2(1- \alpha)}>1$. Since the mapping  $w\mapsto w^{\frac{1}{2(1- \alpha)}}$ is convex, for $t>t_3$, we obtain further that
\begin{align*}
&\beta_{t+1} \le \beta_t - c_3c_4\left(F(x^t) - \bar F^* + \frac{1}{2(\gamma + \theta)}{\rm dist}^2(x^t,S)\right)^{\frac{1}{2(1- \alpha)}}\\
&=\beta_t  - c_3c_4\beta_t^{\frac{1}{2(1- \alpha)}}=\beta_t\bigg(1  - c_3c_4\beta_t^{\frac{1}{2(1- \alpha)}-1}\bigg),
\end{align*}
where $c_4:=2^{-{\frac{1}{2(1- \alpha)}}}$.
Since $\frac{1}{2(1- \alpha)}-1 = \frac{2\alpha - 1}{2(1- \alpha)}>0$, using the above inequality and  \cite[Lemma~4.1]{BoLiYao14}, we have
\begin{align}\label{Bolemma41}
  \beta_t\le \bigg(\beta_{t_3+1}^{-\frac{2\alpha-1}{2(1- \alpha)}}\! + \!\frac{2\alpha - 1}{2(1- \alpha)}c_3c_4(t-t_3-1)\bigg)^{-\frac{2(1-\alpha)}{2\alpha-1}} {\rm \ for\ }t>t_3.
\end{align}
Then, for any $t> t_3$ and $t'\ge 0$, we have
  \begin{align*}
  &\|x^{t} -x^{t+t'}\|^2\le 2\left(\|x^{t} - \bar{x}^{t}\|^2 + \|\bar{x}^{t} - x^{t+t'}\|^2\right)\\
  &\stackrel{{\rm (a)}}{\le}   2\left(\|x^t - \bar{x}^t\|^2 + \| \bar{x}^t -x^t\|^2   +2(\gamma +\theta)\left( F(x^t) - F(x^{t+t'})\right)\right)\\
  &= 2\left(2{\rm dist}^2(x^t,S)  +2(\gamma +\theta)\left( F(x^t) - F(x^{t+t'})\right)\right)\\
  &\stackrel{{\rm (b)}}{\le} 2\left(2{\rm dist}^2(x^t,S)  +4(\gamma +\theta)\left( F(x^t) - \bar F^*\right)\right)\\
  &\stackrel{{\rm (c)}}{=}  8(\gamma + \theta)\beta_t
  \le  8(\gamma + \theta)\left(\beta_{t_3+1}^{-\frac{2\alpha-1}{2(1- \alpha)}} + \frac{2\alpha - 1}{2(1- \alpha)}c_3c_4(t-t_3-1)\right)^{-\frac{2(1-\alpha)}{2\alpha-1}},
  \end{align*}
  where (a) follows from \eqref{summable} and the first equality uses the definition of $\bar x^t$ (i.e., the projection of $x^t$ onto $S$), (b) follows from Lemma~\ref{constantbarForF}(ii), (c) uses the definition of $\beta_t$ and the last inequality follows from \eqref{Bolemma41}.
  Letting $t'\to\infty$ and recalling that $x^t \to x^*$, we see that the conclusion in (ii) holds. This completes the proof.
\end{proof}
\begin{remark}
  From the proof of the above theorem, we can actually deduce that the sequence
  $\left\{F(x^t) - \bar F^* + c_0{\rm dist}^2(x^t,S)\right\}$ (with some suitable $c_0 > 0$) is $Q$-linearly convergent when $F$ is a KL function with exponent $\alpha\in[0,\frac12]$, and is sublinearly convergent when  $F$ is a KL function with exponent $\alpha\in(\frac12,1)$; see \eqref{induction2} and \eqref{Bolemma41}.
\end{remark}

\section{KL properties of $\bar F$ and $F$}
In Section 3, we deduced the rate of convergence of the sequence $\{x^t\}$ generated by SCP$_{ls}$ under nonconvex and convex settings by imposing KL assumptions on $\bar F$ in \eqref{barF} and $F$ in \eqref{Or}, respectively; see Theorem~\ref{bastationKL} and Theorem~\ref{convexKL}. \revise{Note that the assumptions in Theorem~\ref{bastationKL} and Theorem~\ref{convexKL} for \eqref{Or} are different as follows:
\begin{itemize}
  \item Assumptions~\ref{basebase}, \ref{baseass},  \ref{diffofgi} and \ref{p1p2lip} are used in Theorem~\ref{bastationKL}.
  \item Assumptions~\ref{basebase}, \ref{baseass},  \ref{convexassumption} are used in Theorem~\ref{convexKL}.
\end{itemize}
Thus, it is interesting to find a relationship between KL exponent of $\bar F$ and that of $F$ when all the above assumptions hold. In this regard, we have the following theorem.}
\begin{theorem}[\textbf{Relation between the KL exponents of $\bar F$ and $F$}]\label{KLbarFKLF}
Let $F$ be defined as in \eqref{Or} and suppose that Assumptions~\ref{basebase}, \ref{baseass}, \ref{diffofgi} and \ref{convexassumption} hold. If $\bar F$ defined in \eqref{barF} is a KL function with exponent $\alpha\in[0,1)$, then $F$ is also a KL function with exponent $\alpha$.
\end{theorem}
\begin{proof}
  Fix any $x_0\in {\rm dom} \partial F$ and $w_0\in\R$.
  Using \eqref{partial} and noting that $P_2= 0$ (Assumption~\ref{convexassumption}), we have
  for any $x\in{\rm dom}\partial F$ that
  \begin{equation}\label{subdifsubeq}
  \begin{split}
  &\partial \bar F(x,x,w_0) \\
  &\supseteq\left\{ \begin{pmatrix}
    \nabla f(x) +\partial P_1(x)+ \sum_{i=1}^m\lambda_i\nabla g_i(x)\\
    0\\
    0
  \end{pmatrix}:\, \ \lambda \in N_{-\R^m_+}(\bar G(x,x,w_0))\right\}\\
  & \stackrel{{\rm (a)}}{=} \left\{ \begin{pmatrix}
    \nabla f(x) + \partial P_1(x) + \sum_{i=1}^m\lambda_i\nabla g_i(x)\\
    0\\
    0
  \end{pmatrix}:\ \lambda \in N_{-\R^m_+}(g(x))\right\}\\
  &\stackrel{{\rm (b)}}{=} \begin{pmatrix}
   \nabla f(x) +  \partial P_1(x) +  N_{g(\cdot)\le0}(x)\\
    0\\
    0
  \end{pmatrix} \stackrel{{\rm (c)}}{=} \begin{pmatrix}
   \partial F(x) \\
    0\\
    0
  \end{pmatrix},
  \end{split}
  \end{equation}
 where (a) follows from the fact that  $g(x) = \bar G(x,x,w_0)$, (b) follows from Assumption~\ref{baseass} and \cite[Theorem~6.14]{RoWe97}, and (c) holds due to \cite[Exercise~8.8]{RoWe97} and \cite[Theorem~23.8]{Ro70} together with the  convexity of $P_1$ and  $g$ and the continuity of $P_1$.  Using this together with the assumption that $x_0\in {\rm dom}\partial  F$, we have $ (x_0,x_0,w_0)\in {\rm dom}\partial  \bar F $. Then, from the KL assumption on $\bar F$, we see that there exist $a>0$, $\epsilon>0$ and $c_0>0$ such that
 \begin{align}\label{barFFKL}
  {\rm dist}(0,\partial \bar F(x,y,w))\ge a(\bar F(x,y,w) - \bar F(x_0,x_0,w_0))^\alpha
  \end{align}
  whenever $0<\bar F(x,y,w)-\bar F(x_0,x_0,w_0)< c_0$ and $\|(x,y,w) - (x_0,x_0,w_0)\|\le \epsilon$.

In addition, thanks to the fact that $g(x) = \bar G(x,x,w_0)$,   for any $x\in{\rm dom}\partial F$ satisfying $F(x_0)<F(x)<F(x_0) + c_0$, we have
\begin{align}\label{barFx0}
\bar F(x_0,x_0,w_0)<\bar F(x,x,w_0)<\bar F(x_0,x_0,w_0)+ c_0.
 \end{align}
 On the other hand, for $x$ such that $\|x - x_0\|\le \frac12\epsilon$, we have $\|(x,x,w_0) - (x_0,x_0,w_0)\|\le \epsilon$. Using this and \eqref{barFx0}, for  $x\in{\rm dom}\partial F$ satisfying $\|x - x_0\|\le \frac12\epsilon$ and $F(x_0)<F(x)<F(x_0) + c_0$, we have
\begin{align*}
{\rm dist}(0,\partial F(x)) &\overset{\rm (a)}\ge {\rm dist}(0,\partial\bar F(x,x,w_0))\overset{\rm (b)}\ge a(\bar F(x,x,w_0) - \bar F(x_0,x_0,w_0))^\alpha \\
&\overset{\rm (c)}= a(F(x) - F(x_0)\!)^\alpha,
\end{align*}
  where  (a) follows from \eqref{subdifsubeq}, (b) uses \eqref{barFFKL}  and (c) holds thanks to $g(x) = \bar G(x,x,w_0)$. This completes the proof.
\end{proof}

\subsection{KL exponent for a concrete model}\label{KLofFconvex}
 In this subsection, we study the KL exponent of $F$ in \eqref{Or} with additional assumptions on the functions involved. Specifically, we consider
 the following multiply constrained optimization problem:
\begin{align}\label{convexF}
  \min_{x\in\R^n}  P_1(x) + \delta_{g(\cdot)\le 0}(x),
\end{align}
where  $P_1$ is convex continuous,  the function  $g(x) = (l_1(A_1x),\dots,l_m(A_mx))$ with each $A_i\in\R^{q_i\times n}$ and  $l_i:\R^{q_i}\to\R$ being {\em strictly} convex, and $\{x:\; g(x)\le 0\}\neq \emptyset$. Clearly, \eqref{convexF} is a special case of \eqref{Or} with $f=P_2= 0,  \ g_i(x) = l_i(A_ix), {\rm \ for\  }i=1,\dots,m$ and
\begin{align}\label{convexFF}
F(x) = P_1(x) + \delta_{g(\cdot)\le 0}(x) \revise{= P_1(x) + \sum_{i=1}^m\delta_{l_i(\cdot)\le 0}(A_ix)}.
\end{align}
We will derive rules to deduce the KL exponent of $F$ in \eqref{convexFF} from its Lagrangian. Similar rules were introduced in  \cite{LiPo18} and \cite{YuLiPo19}, which studied the KL exponent of $F$ in \eqref{convexFF} respectively  when $m = 1$ and when the constraint set is defined by equality constraints, under suitable assumptions. Here, we look at \eqref{convexFF} that involves multiple inequality constraints.
\begin{theorem}[{{\bf KL exponent of \eqref{convexFF} from its Lagrangian}}]\label{KLofconvexKL}
 Let $F$ be  as in \eqref{convexFF} and  $\bar x\in \Argmin F $. Suppose the following conditions  hold:
 \begin{enumerate}[{\rm(i)}]
   \item There exists a Lagrange multiplier $\bar \lambda\in\R_+^m$ for \eqref{convexF}
        and   $x\mapsto P_1(x) + \langle\bar \lambda, g(x)\rangle$ is a KL function with exponent $\alpha\in(0,1)$.
   \item  The strict complementarity condition holds at $(\bar x,\bar \lambda)$, i.e., for every $i$ satisfying $\bar \lambda_i=0$, it holds that   $l_i(A_i\bar x) <0$.
 \end{enumerate}
 Then $F$ satisfies the  KL property with exponent $\alpha$ at $\bar x$.
\end{theorem}
\begin{proof}
  Let $ F_{\bar \lambda}(x):=P_1(x) + \langle\bar \lambda, g(x)\rangle$. By the definition of Lagrange multiplier, we have
  \begin{align}\label{equalityff'}
 F(\bar x) =\inf F =P_1(\bar x) = \inf F_{\bar \lambda}\le F_{\bar \lambda}(\bar x)\le F(\bar x),
  \end{align}
  where the second  inequality holds because $g(\bar x)\le 0$ and $\bar \lambda\in \R_+^m$. On the other hand, thanks to (ii), it holds that $\{i:\bar \lambda_i> 0\} = I(\bar x)$. This together with \cite[Theorem~28.1]{Ro70}  gives
  \begin{align}\label{argmin}
  \bar x\in\Argmin F =\bigcap_{i\in I(\bar x)}\!\!\{x\!:\;\! l_i(A_ix)=0\}\cap\bigcap_{i\not\in I(\bar x)}\!\!\{x:\; l_i(A_ix)\le 0\}\cap\Argmin F_{\bar \lambda}.
  \end{align}
  Since $l_i$ is strictly convex and $\bar \lambda_i>0$ for $i\in I(\bar x)$, we see that $A_ix$ is constant over $\Argmin F_{\bar \lambda}$ for each  $i\in I(\bar x)$. This together with the fact that $l_i(A_i\bar x) = 0$ for $i\in I(\bar x)$ and \eqref{argmin}  implies  that
  \begin{align}\label{noball}
  \bar x\in\Argmin F = \bigcap_{i\not\in I(\bar x)}\{x:\; l_i(A_ix)\le 0\}\cap\Argmin F_{\bar \lambda}.
  \end{align}
  Next, since $l_i(A_i\bar x)<0$ for each  $i\not\in  I(\bar x)$, there exists  $\epsilon_0>0$ such that
  \[
 l_i(A_ix)<0, \ \forall x\in B(\bar x,\epsilon_0),\  \forall \ i\not\in I(\bar x).
  \]
  This together with \eqref{noball} implies that
  \begin{align}\label{ballargmin}
  \bar x\in\Argmin F\cap B(\bar x,\epsilon_0) = \Argmin F_{\bar \lambda}\cap B(\bar x,\epsilon_0) .
  \end{align}

 Now, using  (i) and \cite[Theorem~5(i)]{BoNgPeSu17} together with the fact that $\bar x\in\Argmin F_{\bar \lambda}$, we see that there exist $\bar a>0$, $\bar c>0$ and $0<\epsilon<\epsilon_0$ such that
 \begin{align}\label{errorbound}
 {\rm dist}(x,\Argmin F_{\bar \lambda} )\le \bar c(F_{\bar \lambda}(x) - F_{\bar \lambda}(\bar x))^{1-\alpha}
 \end{align}
 whenever $\|x - \bar x\|\le \epsilon$ and $F_{\bar \lambda}(\bar x)\le F_{\bar \lambda}(x)<F_{\bar \lambda}(\bar x) + \bar a$. Note that for any $x$ satisfying $F(\bar x)<F(x)<F(\bar x) + \bar a$,
  we have $l_i(A_ix)\le 0$ for each  $i$ and
  \begin{align}\label{equivv}
 F(\bar x) = F_{\bar \lambda}(\bar x) \le F_{\bar \lambda}(x)\le F(x)< F(\bar x)+\bar a = F_{\bar \lambda}(\bar x) + \bar a,
 \end{align}
  where the first and the last equalities follow  from \eqref{equalityff'} and the second inequality holds because $\bar \lambda_i\ge0$ and $l_i(A_ix)\le 0$ for each  $i=1,\dots,m$.
  Therefore, for any $x$ satisfying $F(\bar x)<F(x)<F(\bar x) + \bar a$ and $\|x -\bar x\|\le \epsilon$,  we have
  \begin{equation*}
  \begin{split}
 &{\rm dist}(x,\Argmin F )\le {\rm dist}(x,\Argmin F \cap B(\bar x,\epsilon_0))\stackrel{{\rm (a)}}{=}{\rm dist}(x,\Argmin F_{\bar \lambda}\cap B(\bar x,\epsilon_0) )\\
 &\stackrel{{\rm (b)}}{\le} 4\max\left\{{\rm dist}(x,\Argmin F_{\bar \lambda}),{\rm dist}(x,B(\bar x,\epsilon_0))\right\} \stackrel{{\rm (c)}}{=} 4{\rm dist}(x,\Argmin F_{\bar \lambda})\\
 &\stackrel{{\rm(d)}}{\le} 4\bar c(F_{\bar \lambda}(x) - F_{\bar \lambda}(\bar x))^{1-\alpha} \le 4\bar c(F(x) - F(\bar x))^{1-\alpha},
 \end{split}
  \end{equation*}
    where (a) follows from \eqref{ballargmin}, (b) follows from \cite[Lemma~4.10]{LiNgPo07}, (c) holds because $\epsilon<\epsilon_0$, (d) follows from \eqref{errorbound} and \eqref{equivv}  and the last inequality holds because of \eqref{equalityff'} (so that $F_{\bar\lambda}(\bar x) = F(\bar x)$), $l_i(A_ix)\le 0$ for each  $i$ and $\bar \lambda\in\R_+^m$. The desired conclusion now follows immediately from this and \cite[Theorem~5(ii)]{BoNgPeSu17}.
\end{proof}

Now, we give a corollary that deals with \eqref{convexF} with $m = 1$. This result is different from \cite[Theorem~3.5]{LiPo18} because, here, it is the constraint function that is a composition of strictly convex function and a linear map, but not the objective function.
\begin{corollary}\label{KLofconvexKLcor}
 Let $F$ be defined as in \eqref{convexFF} with $m=1$.  Suppose the following conditions  hold:
 \begin{enumerate}[{\rm(i)}]
   \item It holds that $\inf P_1<\inf F$.
   \item There exists a Lagrange multiplier\footnote{ \revise{Following \cite[Page~274]{Ro70}, we say that $\bar \lambda$ is a Lagrange multiplier for (4.4) if $\bar\lambda\ge 0$ and  $\inf\limits_{x\in \R^n} \{P_1(x) + \bar\lambda g(x)\} = \inf\limits_{x\in \R^n} \{P_1(x) + \delta_{g(\cdot)\le 0}(x)\} > -\infty$.}} $\bar \lambda\ge0$ for  \eqref{convexF}
          and  $x\mapsto P_1(x) + \bar \lambda l_1(A_1x)$ is a KL function with exponent $\alpha\in(0,1)$.
 \end{enumerate}
 Then $F$ is KL function with exponent $\alpha$.
\end{corollary}
\begin{proof}
Let $F_{\bar \lambda}(x):=P_1(x) + \bar \lambda l_1(A_1x)$.
  In view of \cite[Lemma~2.1]{LiPo18} and the convexity of $F$, it suffices to show that $F$ has KL property at every point in $\{x:\;0\in\partial F(x)\} = \Argmin F$ with exponent $\alpha$. Fix any $\bar x$ with $0\in\partial F(\bar x)$.
  Then one can see from condition (i) and the definition of Lagrange multiplier that $\bar \lambda>0$ and thus  $l_1(A\bar x)= 0$.  Therefore, Assumption (ii) of  Theorem~\ref{KLofconvexKL} is satisfied. This together with  (ii)  and  Theorem~\ref{KLofconvexKL} shows that $F$ satisfies the KL property at $\bar x$ with exponent $\alpha$.
\end{proof}

\begin{remark}\label{ells}
   When $P_1(\cdot) = \|\cdot\|_1$ in \eqref{convexF}, we deduce from \cite[Corollary~5.1]{LiPo18} and Corollary~\ref{KLofconvexKLcor} that the KL exponent of $F$ in \eqref{convexFF} is $\frac12$ if $m = 1$ and $l_1$ takes one of the following forms with $b\in\R^q$ and $\delta>0$ chosen so that the Slater condition holds and the origin is not feasible:
  \begin{enumerate}[{\rm (i)}]
    \item  (Basis pursuit denoising \cite{Ca08}) $l_1(z) = \frac12\|z - b\|^2 - \delta$.
    \item  (Logistic loss \cite{HoLeSt13,KlDiGa02}) $l_1(z) = \sum_{i=1}^q\log(1+ \exp(b_iz_i)) - \delta$ for some $b\in\R^q$.
    \item  (Poisson loss \cite{Fr83,La92,Zo04}) $l_1(z) = \sum_{i=1}^q(-b_iz_i + \exp(z_i)) - \delta$ for some $b\in\R^q$.
  \end{enumerate}
\end{remark}

\section{Applications in compressed sensing}\label{applications}

In this section, we consider  applications of \eqref{Or} and discuss how the various assumptions required in our analysis of SCP$_{ls}$ can be verified.
We focus on the problem of compressed sensing, which attempts to reconstruct sparse signals from possibly noisy low-dimensional measurements; see \cite{CaRaArBaSa16} for a recent review. We specifically look
at the following model:
\begin{equation}\label{lOr}
\begin{array}{rl}
\min\limits_{x}& \|x\|_1 - \mu\|x\|\\
{\rm s.t.}& \ell(Ax - b)\le \delta,
\end{array}
\end{equation}
where $\mu\in[0,1]$, $A\in\R^{q\times n}$ has {\em full row rank}, $b\in\R^q$, $\ell:\R^q\rightarrow \R_+$ is an analytic function whose gradient is Lipschitz continuous with modulus $L_{\ell}$ and satisfies $\ell(0)=0$, and  $\delta\in (0,\ell(-b))$. The $\ell$ in \eqref{lOr} is typically chosen according to different types of noise. We will look at two specific choices in Section~\ref{sec5.1} and Section~\ref{sec5.2}, respectively.

Problem \eqref{lOr} is a special case of \eqref{Or} with $f= 0$, $P_1(x) = \|x\|_1$, $P_2(x) = \mu\|x\|$ and $g(x) = \ell(Ax - b) - \delta$.\footnote{Note that $\{x:\; g(x)\le0\}\neq \emptyset$ because $A$ has full row rank and $\ell(0) = 0 < \delta$.} Then the $F$ from \eqref{Or} corresponding to \eqref{lOr} is
\begin{equation}\label{Felll}
  F(x) = \|x\|_1 - \mu\|x\| + \delta_{\ell(A\cdot - b)\le \delta}(x),
\end{equation}
and the $\bar F$ from \eqref{barF} corresponding to \eqref{lOr} is
 \begin{align}\label{barFell1}
\bar F(x,y,w) =  \|x\|_1 - \mu\|x\| + \delta_{\bar  G(\cdot)\le 0}(x,y,w)
\end{align}
with
\begin{equation}\label{barGell1}
\bar  G(x,y,w) = \ell(Ay-b) + \langle A^T\nabla \ell(Ay-b),x-y\rangle + \frac{w}2\|x-y \|^2 - \delta.
\end{equation}
Our next theorem concerns the KL conditions needed in Theorems~\ref{bastationKL} and \ref{convexKL}.
\begin{theorem}\label{subana}
  Let $F$ and $\bar F$ be defined as in \eqref{Felll} and \eqref{barFell1}, respectively, \rerevise{and let $\Xi\subseteq {\rm dom}\,\partial F$ and $\Upsilon\subseteq {\rm dom}\,\partial \bar F$ be compact sets. Then there exists $\alpha\in [0,1)$ so that $F$ (resp., $\bar F$) satisfies the KL property with exponent $\alpha$ at every point in $\Xi$ (resp., in $\Upsilon$).}
\end{theorem}
\begin{proof}
Let ${\frak D}_0 := \{x:\;\ell(Ax-b)\le \delta\}$ and $\mathfrak{D}_1 = \{(x,y,w):\; \bar G(x,y,w)\le 0\}$, where $\bar G$ is as in \eqref{barGell1}. Since $\ell$ and $\bar G$ are analytic, we have that ${\frak D}_0$ and $\mathfrak{D}_1$ are semianalytic; see \cite[Page~596]{FacPan03} for the definition.

   On the other hand, since $x\mapsto\|x\|_1 -\mu\|x\|$ is semialgebraic, it holds that ${\frak F}_0:= \{(x,z):\; z=\|x\|_1 - \mu\|x\|\}$ and $\mathfrak{F}_1:=\{(x,y,w,z):\;z=\|x\|_1 - \mu\|x\|\}$ are subanalytic (see \cite[Page~597(p2)]{FacPan03} for the subanalyticity of ${\frak F}_1$). Therefore,
   \[
   {\rm gph}(F) = \mathfrak{F}_0\cap ({\frak D}_0\times \R)\ \ {\rm and}\ \ {\rm gph}(\bar F) = \mathfrak{F}_1\cap(\mathfrak{D}_1\times \R)
   \]
   are subanalytic, thanks to \cite[Page~597(p1)\&(p2)]{FacPan03}. Also, the functions $F$ and $\bar F$ have closed domains and are continuous on their respective domains. Thus, the desired conclusion follows from \cite[Theorem~3.1]{BoDaLe07} \rerevise{and a standard compactness argument as in the proof of \cite[Lemma~1]{AtBo09}}.
\end{proof}

We next focus on two common choices of $\ell$ in \eqref{lOr}: $\ell(\cdot)=\frac12\|\cdot\|^2$ (for Gaussian noise  \cite{vaFr09}) and $\ell(\cdot) = \|\cdot\|_{LL_2,\gamma}$  being the Lorentzian norm (for Cauchy noise \cite{CaBaAy10}) for some $\gamma>0$. We will discuss how to verify the other assumptions necessary for the applications of Theorem~\ref{bastationKL} or Theorem~\ref{convexKL} to \eqref{lOr} with these two choices of $\ell$.

\subsection{When $\ell(\cdot)=\frac12\|\cdot\|^2$}\label{sec5.1}
In this case, the \revise{model \eqref{lOr}} becomes
\begin{equation}\label{Leastsq}
\begin{array}{rl}
  \min\limits_{x}& \ \|x\|_1 - \mu\|x\|\\
  {\rm s.t.}&\  \frac12\|Ax - b\|^2\leq\delta,
  \end{array}
\end{equation}
and the corresponding $F$ in \eqref{Or} becomes:
\begin{align}\label{Fleast}
F(x)=\|x\|_1 - \mu\|x\| +\delta_{g(\cdot)\le 0}(x),
\end{align}
with $f= 0$, $P_1(x) = \|x\|_1$, $P_2(x) = \mu\|x\|$ and   $g(x)=\frac12\|Ax -b\|^2 - \delta$ for  $A$, $b$, $\delta$ and $\mu$ as in \eqref{lOr}. Then, for \eqref{Leastsq}, $P_1$ and $P_2$ are convex continuous, and Assumption~\ref{basebase}(i) and (ii) and Assumption~\ref{diffofgi} are satisfied. Moreover, $A$ having full row rank and $\delta\in(0,\frac12\|b\|^2)$ imply that Slater condition holds for \eqref{Leastsq}. Hence, it \rerevise{follows} that $\{x:\; g(x)\le 0\}\neq \emptyset$ \rerevise{and Assumption~\ref{baseass} holds}. In addition, this $P_2$ satisfies Assumption~\ref{p1p2lip} since its only possible point of nondifferentiability (the origin) is not feasible thanks to the fact that $\delta<\frac12\|b\|^2$. Furthermore, the required KL conditions follow from Theorem~\ref{subana}.\footnote{\rerevise{Specifically, if $\mu = 0$, then $F$ is convex and level-bounded, and the set of stationary points (minimizers) is compact. We can then deduce from Theorem~\ref{subana} that $F$ is a KL function with some exponent $\alpha\in [0,1)$. On the other hand, the KL property required in the nonconvex case (see Theorem~\ref{bastationKL}) follows directly from Theorem~\ref{subana}.}}
In order to apply Theorem~\ref{bastationKL} (or Theorem~\ref{convexKL}), we now demonstrate how conditions can be imposed so that Assumption~\ref{basebase}(iii) (level-boundedness) is satisfied.
 \begin{proposition}\label{Prop5.1}
 Let $F$ be defined as in \eqref{Fleast}. The following statements hold:
   \begin{enumerate}[{\rm (i)}]
     \item If $\mu\in[0,1)$, then $F$ is level-bounded.
     \item If $\mu=1$ and $A$ does not have zero columns, then $F$ is level-bounded.
   \end{enumerate}
 \end{proposition}
 \begin{proof}
 Note first that if $0\le\mu<1$, then $x\mapsto \|x\|_1 - \mu\|x\|$ is level-bounded and hence (i) holds trivially. We next focus on the case where $\mu=1$.

 Suppose to the contrary that there exists $\sigma$ and  $\{x^t\}\subseteq \left\{x:\;F(x) \leq\sigma \right\} $ such that  $\|x^t\|\to \infty$. By passing to a further subsequence if necessary, we may assume that there exists $d$ with $\|d\| = 1$ and $\lim\limits_{t\to \infty}\frac{x^t}{\|x^t\|} = d$. Since $\frac12\|Ax^t - b\|^2\leq \delta$ thanks to $F(x^t)\le \sigma$ for each $t$, we have
\begin{equation}\label{Ad0}
\begin{split}
\frac12\|Ad\|^2=\lim_{t\to\infty}\frac12\frac{\|Ax^t - b\|^2}{\|x^t\|^2} \leq \lim_{t\to \infty}\frac{\delta}{\|x^t\|^2} = 0.
\end{split}
\end{equation}
On the other hand, since $F(x^t)\le \sigma$, it holds that
\[
0\le\|x^t\|_1 - \|x^t\|\leq\sigma \Longrightarrow 0\le\lim_{t\to\infty}\frac{\|x^t\|_1 - \|x^t\|}{\|x^t\|}  = \|d\|_1 - 1\le0.
\]
This together with $\|d\|=1$ implies that  exactly one  coordinate of $d$ is nonzero.
Since $A$ does not have zero columns, we obtain that $\|Ad\|\neq 0$, which contradicts \eqref{Ad0}. Thus, the statement in (ii) holds.
 \end{proof}

Therefore, if the assumptions in the above proposition hold, one can apply Theorem~\ref{bastationKL} or Theorem~\ref{convexKL} to deducing the convergence rate of the sequence generated by SCP$_{ls}$ when applied to solving \eqref{Leastsq}.
When $\mu=0$ in \eqref{Leastsq}, since we assumed $\delta \in (0,\frac12\|b\|^2)$ and $A$ has full row rank, we know from Remark \ref{ells} that  $x\mapsto \|x\|_1 + \delta_{ \frac12\|A(\cdot) - b\|^2 \le \delta}(x)$ is a KL function with exponent $\frac12$. Consequently, the sequence $\{x^t\}$ generated by SCP$_{ls}$ for \eqref{Leastsq} converges locally linearly.
When $\mu \in(0,1]$, although no explicit KL exponent is known for the corresponding $\bar F$, we still observe in our numerical experiments below that the sequence $\{x^t\}$ generated by SCP$_{ls}$ for \eqref{Leastsq} appears to converge linearly.

\subsection{When $\ell$ is the Lorentzian norm}\label{sec5.2}
Recall that, given $\gamma > 0$, the Lorentzian norm of a vector $y\in\R^q$ is defined as
\[
\|y\|_{LL_{2},\gamma} := \sum_{i=1}^q\log\left(1 + \frac{y_i^2}{\gamma^2}\right).
\]
In this case, the model \eqref{lOr} becomes
\begin{equation}\label{LL2}
\begin{split}
  \min_{x} &\ \|x\|_1 - \mu\|x\|\\
  {\rm s.t.}& \ \|Ax - b\|_{LL_{2},\gamma}\leq\delta,
  \end{split}
\end{equation}
and the corresponding $F$ in \eqref{Or} now takes the following form:
\begin{align}\label{FLL2}
F(x)=\|x\|_1 - \mu\|x\| +\delta_{g(\cdot)\le 0}(x),
\end{align}
with $f= 0$, $P_1(x) = \|x\|_1$, $P_2(x) = \mu\|x\|$ and   $g(x)=\|Ax -b\|_{LL_2,\gamma} - \delta$ for  $A$, $b$, $\delta$ and $\mu$  defined as in \eqref{lOr}.
One can show that the mapping $ z\mapsto\|z\|_{LL_{2},\gamma}-\delta$ has Lipschitz gradient with modulus $L_{\ell} = \frac{2}{\gamma^2}$ and is twice continuously differentiable.
From these one can readily see that $P_1$ and $P_2$ are convex continuous, and Assumption~\ref{basebase}(i) and (ii) and Assumption~\ref{diffofgi} are satisfied. Also, since $A$ has full row rank and $\delta\in(0,\|b\|_{LL_2,\gamma})$, we see that $\{x:\; g(x)\le 0\}\neq \emptyset$. In addition, this $P_2$ satisfies Assumption~\ref{p1p2lip} since its only possible point of nondifferentiability is not feasible, thanks to $\delta\in(0,\|b\|_{LL_2,\gamma})$.
Furthermore, the required KL conditions follow from Theorem~\ref{subana}.
In order to apply Theorem \ref{bastationKL}, we show below that Assumption~\ref{baseass} holds and impose conditions so that Assumption~\ref{basebase}(iii) is satisfied.
\begin{proposition}
 Let $F$ be defined as in \eqref{FLL2}. The following statements hold:
 \begin{enumerate}[{\rm (i)}]
   \item The MFCQ holds in the whole feasible set of \eqref{LL2}.
   \item If $\mu\in[0,1)$, then $F$ is level-bounded.
   \item If $\mu=1$ and $A$ does not have zero columns, then $F$ is level-bounded.
 \end{enumerate}
 \end{proposition}
 \begin{proof}
 For (i), using the definition of MFCQ, it suffices to show that for every feasible point $x$ with $g(x) = 0$, it holds that $\nabla g(x) \neq 0$. Suppose to the contrary that there exists $\hat x$ such that $g(\hat x) =0$ and
\[
 \nabla g(\hat x)  = A^T\left(\frac{2(A\hat x - b)_1}{\gamma^2+(A\hat x -b)^2_1},\dots,\frac{2(A\hat x - b)_q}{\gamma^2+(A\hat x -b)^2_q}\right)^T=0.
\]
Since $A$ is surjective, we deduce that $\left(\frac{2(A\hat x - b)_1}{\gamma^2+(A\hat x -b)^2_1},\dots,\frac{2(A\hat x - b)_q}{\gamma^2+(A\hat x -b)^2_q}\right) = 0$. This shows that $A\hat x - b = 0$ and thus $g(\hat x) = \|A\hat x - b\|_{LL_2,\gamma} - \delta = -\delta\neq 0$, a contradiction. Therefore, the MFCQ holds in the whole feasible set of \eqref{LL2}.

The assertion in (ii) holds trivially. We now prove (iii).
Suppose to the contrary that there exist $\sigma$ and $\{x^t\}\subseteq \left\{x:\;F(x) \leq\sigma \right\} $ such that $\|x^t\|\to \infty$. By passing to a further subsequence if necessary, we may assume that there exists $d$ with $\|d\| = 1$ and $d = \lim\limits_{t\to \infty}\frac{x^t}{\|x^t\|}$. Since $\ell(Ax^t - b)\le 0$ thanks to $F(x^t)\le \sigma$ for each  $t$, and the Lorentzian norm is level-bounded, we see that there exists $\xi$ such that
$\|Ax^t - b\|\le \xi$ for all $t$. The rest of the proof is then the same as that of Proposition~\ref{Prop5.1}(ii).
\end{proof}

Therefore, if the assumptions in the above proposition hold, one can apply Theorem~\ref{bastationKL} to deducing the convergence rate of the sequence $\{x^t\}$ generated by SCP$_{ls}$ when applied to solving \eqref{LL2}.
Although no explicit KL exponent is known for the corresponding $\bar F$, in our numerical experiments below, we observe empirically that the sequence $\{x^t\}$ generated by SCP$_{ls}$ for \eqref{LL2} appears to converge linearly.

\subsection{Numerical experiments}
\revise{In this subsection, we perform numerical experiments to illustrate the convergence results  of SCP$_{ls}$ established in Section \ref{sec3}. We apply SCP$_{ls}$ to \eqref{lOr} with $\ell$ being either $\frac12\|\cdot\|^2$ (as in \eqref{Leastsq}) or the Lorentzian norm (as in \eqref{LL2}). We also consider the SCP in \cite{Lu10} in our experiments below.}

\paragraph{\bf Algorithms and their parameters}
We consider the following algorithms:
\begin{enumerate}[{\rm (i)}]
\item\textbf{SCP$_{ls}$}:
We solve the corresponding subproblem \eqref{subp} through a root-finding scheme outlined in Appendix A.
  Moreover, we let $\tau = 2$, $c = 10^{-4}$, ${\b L} = 10^{-8}$, ${\rm \bar L} = 10^8$.
For $t=0$, we choose $L_f^{t,0} =1$ and $L_{g}^{t,0} = 1$. For $t\ge 1$, we choose:
\begin{align*}
L_f^{t,0} =1,\ L_{g}^{t,0} =\begin{cases} \max\left\{10^{-8}, \min\left\{\frac{\langle\Delta x,\Delta g\rangle}{\|\Delta x\|^2}, 10^8\right\}\right\}  &{\rm if\ } \langle\Delta x,\Delta g\rangle \ge 10^{-12},\\
                             \rerevise{\max\left\{10^{-8}, \min\left\{L_g^{t-1}/\tau, 10^8\right\}\right\} } &{\rm else},
                             \end{cases}
\end{align*}
where $\Delta x  = x^t - x^{t-1}$ and $\Delta g = \nabla g(x^t) -\nabla g(x^{t-1})$. We initialize SCP$_{ls}$ at $A^\dagger b$ and terminate it when $\|x^{t+1} - x^t\|< 10^{-8}\max\{1,\|x^{t+1}\|\}$.
  \item\textbf{SCP}: \revise{This was proposed in \cite{Lu10}.} The subproblem of SCP is solved using  a root-finding scheme outlined in Appendix A. We initialize SCP at $A^\dagger b$ and terminate it when $\|x^{t+1} - x^t\|< 10^{-8}\max\{1,\|x^{t+1}\|\}$.

\end{enumerate}

\paragraph{\bf Numerical results}
All codes are written in Matlab, and the experiments are performed in Matlab \revise{2019b} on a 64-bit PC with an Intel(R) Core(TM) i7-4790 CPU (3.60GHz) and 32GB of RAM.

\revise{ For both models \eqref{Leastsq} and \eqref{LL2}, we consider either $\mu = 0$ or $1$.}
In our tests, we let $q = 720i$ and $n = 2560i$ with $i=5$. We generate an $A\in\R^{q\times n}$ with i.i.d standard Gaussian entries, and then normalize this matrix so that each column of $A$ has unit norm. Then we choose a subset $T$ of size $s_0 = [\frac{q}{9}]$ uniformly at random from $\{1,2,\dots,n\}$ and an $s_0$-sparse vector $x_{\rm orig}$ having i.i.d. standard Gaussian entries on $T$ is generated.

  For  \eqref{Leastsq}, we let $b = Ax_{\rm orig} + 0.01\cdot \hat{n}$ with $\hat{n}\in\R^q$ being a random vector with i.i.d. standard Gaussian entries. We then set the  $\delta$ in \eqref{Leastsq} to be $\frac12\sigma^2$ with $\sigma = 1.1\| 0.01\cdot \hat{n}\|$.

  For \eqref{LL2}, we let $b = Ax_{\rm orig} + 0.01\cdot \bar n$ with $\bar n_i\sim {\rm Cauchy(0,1)}$, i.e., $\bar n_i := \tan(\pi(\widetilde{n}_i - 1/2))$ with $\widetilde n\in\R^m$ being a random vector with i.i.d. entries uniformly chosen in $[0,1]$. We set the $\delta$ in \eqref{LL2} to be $1.1\|0.01\bar n\|_{LL_2,\gamma}$ with $\gamma = 0.02$.

 We compare the approximate solution obtained by SCP$_{ls}$ and the original sparse solution in Figures~\ref{figure1} and \ref{figure2} to illustrate the recovery  ability of  SCP$_{ls}$. In Figures~\ref{figure3} and \ref{figure4}, we plot $\|x^t- x^{out}\|$ (in logarithmic scale) against the number of iterations, where $x^t$ and $x^{out}$ are respectively the $t^{\rm th}$ iterate and the approximate solution obtained by the algorithm under study. As we can see, SCP$_{ls}$ always appears to converge linearly and is also faster than \revise{SCP}.

\begin{figure}[tbhp]
\centering
\caption{Recovery results by solving model \eqref{Leastsq} with $\mu=0$ (left) and $\mu = 1$ (right) via SCP$_{ls}$. The approximate solution obtained by SCP$_{ls}$ is marked by asterisk, and $x_{\rm orig}$ is marked by circle.}\label{figure1}
\subfloat{\includegraphics[scale = 0.3]{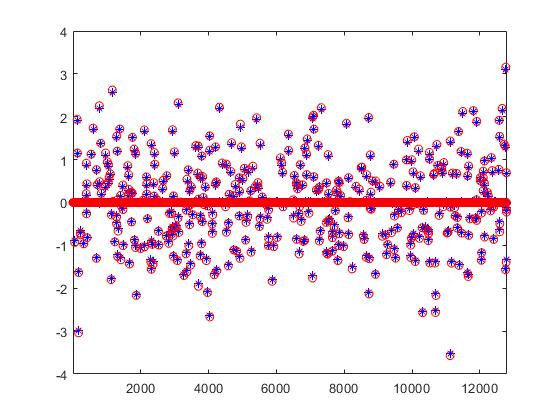}}
\subfloat{\includegraphics[scale = 0.3]{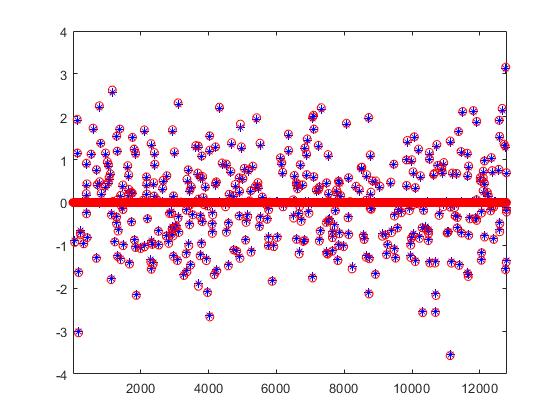}}
\centering
\caption{Recovery results by solving model \eqref{LL2}   with $\mu=0$ (left) and $\mu = 1$ (right) via SCP$_{ls}$. The approximate solution obtained by SCP$_{ls}$ is marked by asterisk, and $x_{\rm orig}$ is marked by circle.}\label{figure2}
\subfloat{\includegraphics[scale = 0.3]{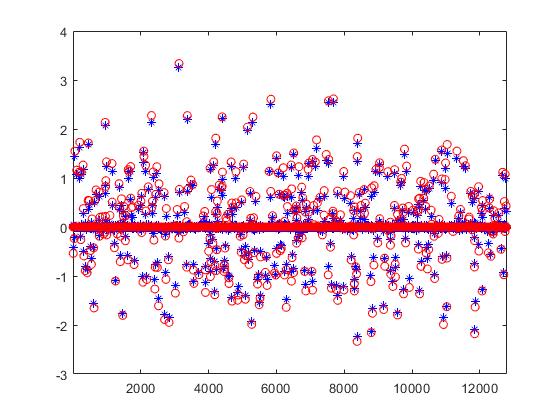}}
\subfloat{\includegraphics[scale = 0.3]{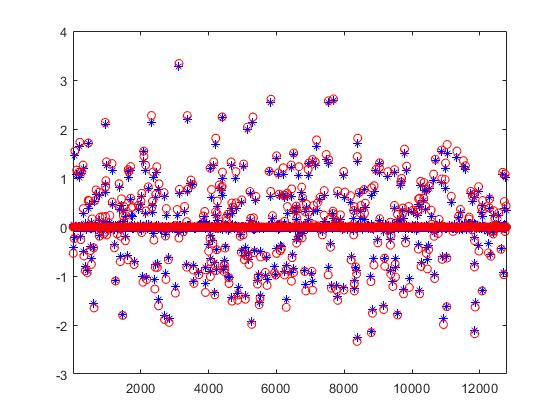}}
\end{figure}

\begin{figure}[tbhp]
\centering
\caption{Plot of $\|x^t-x^{\rm out}\|$ (in log scale) for model \eqref{Leastsq}   with $\mu=0$ (left) and $\mu = 1$ (right). The number in the parenthesis is the CPU time taken.}\label{figure3}
\subfloat{\includegraphics[scale = 0.32]{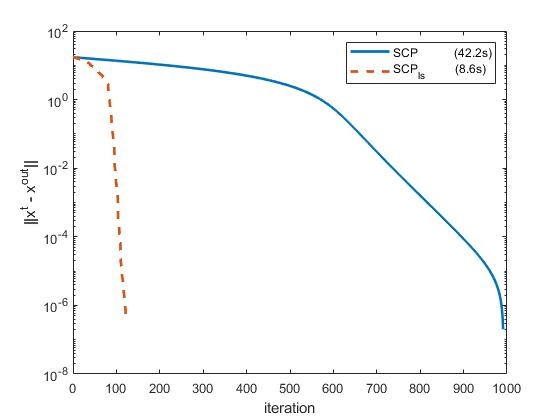}}
\subfloat{\includegraphics[scale = 0.32]{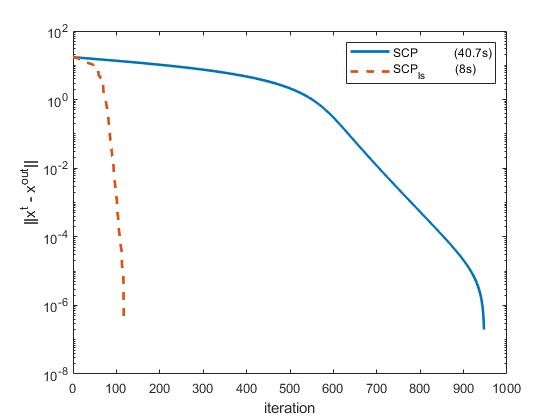}}
\centering
\caption{Plot of $\|x^t-x^{\rm out}\|$ (in log scale) for model \eqref{LL2}   with $\mu=0$ (left) and $\mu = 1$ (right). The number in the parenthesis is the CPU time taken.}\label{figure4}
\subfloat{\includegraphics[scale = 0.32]{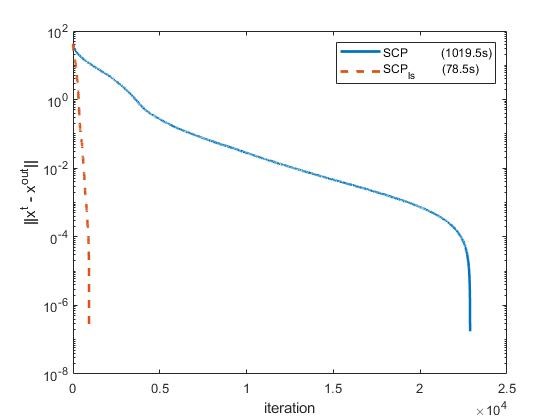}}
\subfloat{\includegraphics[scale = 0.32]{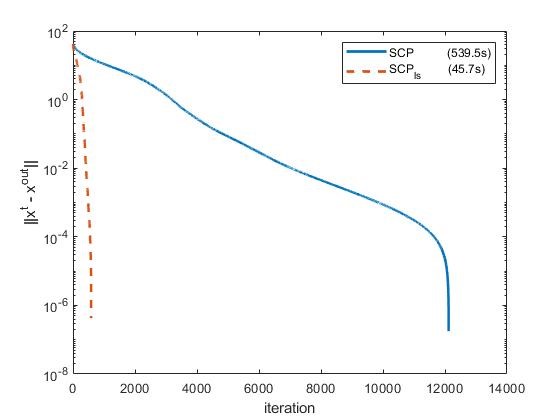}}
\end{figure}

\appendix

\section{Solving the subproblem of SCP$_{ls}$ with $P_1$ being the $\ell_1$ norm, $P_2= 0$ and $m = 1$}

We discuss how the subproblem \eqref{subp} that arises in our numerical tests when SCP$_{ls}$ is applied to \eqref{lOr} can be solved efficiently. Our approach is based on a root-finding strategy for solving the dual, which was also adopted in \cite{ShTe16} for solving the subproblem that arises in the MBA variant there. Comparing with the subproblem considered in \cite{ShTe16}, our subproblem has an additional quadratic term, which slightly complicates the derivation and implementation.

At the $t^{\rm th}$ iteration, the corresponding subproblem \eqref{subp} that arises when SCP$_{ls}$ is applied to \eqref{lOr} takes the following form:
\begin{equation}\label{subproblem}
  \begin{array}{rl}
    \min\limits_{x} & \|x\|_1 + \frac\alpha2\|x - y\|^2\\
    {\rm s.t.} & \|x - s\|^2 \le r,
  \end{array}
\end{equation}
where $y$, $s\in \R^n$, $\alpha> 0$ and $r > 0$.\footnote{The fact that $r > 0$ follows from Lemma~\ref{rdayu0}(iii).}

Recall that the Lagrangian function for \eqref{subproblem} is given by
\begin{align*}
\widetilde L(x,\lambda)= \|x\|_1 + \frac{\alpha}{2}\|x-y\|^2 + \lambda(\|x- s\|^2 -r).
\end{align*}
Using \cite[Corollary~28.2.1, Theorem~28.3]{Ro70}, we know that there exists $(x^*,\lambda^*)$ with $\lambda^* \ge 0$ such that $x^*$ is optimal for \eqref{subproblem} and
\[
\min_{x\in\R^n} \widetilde L(x,\lambda^*) = \min_{x\in\R^n} \|x\|_1 + \frac{\alpha}{2}\|x-y\|^2 + \delta_{\|(\cdot)- s\|^2 \le r }(x).
\]
If $\lambda^* = 0$, then the solution $\check x$ of $\min\limits_{x\in\R^n}  \|x\|_1 +\frac\alpha2\|x-y\|^2$
lies in $\{x:\;\| x  - s\|^2\le r\}$ and $\check x$ solves \eqref{subproblem}. Moreover, $\check x$ is given explicitly as ${\rm sign}(y)\circ\max\{|y| - \frac1\alpha,0\}$, where $\circ$ denotes the entrywise product, and the sign function, absolute value and maximum are taken componentwise.

If $\lambda^*>0$, using \cite[Theorem~28.3]{Ro70}, we obtain that
\begin{equation}\label{A2}
0\in \partial \|x^*\|_1 + \alpha(x^*-y) + 2\lambda^*(x^*-s)\ {\rm and}\ \|x^*-s\|^2 = r.
\end{equation}
Using the first relation in \eqref{A2}, we have
\begin{align}\label{A3}
x^*= {\rm Prox}_{\frac{1}{\alpha + 2\lambda^*}\|\cdot\|_1}\left(\frac{\alpha}{\alpha + 2\lambda^*}y + \frac{2\lambda^*}{\alpha + 2\lambda^*}s\right),
\end{align}
where ${\rm Prox}_h(u):=\argmin\limits_{v\in\R^n}\left\{h(v) + \frac12\|u - v\|^2\right\}$ for a proper closed convex function $h$.
Plugging this into the second relation in \eqref{A2}, we see that $\lambda^*$ can be obtained by solving the following one-dimensional nonsmooth equation and the solution $x^*$ can then be recovered via \eqref{A3}:
\[
\left\|{\rm Prox}_{\frac{1}{\alpha + 2\lambda^*}\|\cdot\|_1}\left(\frac{\alpha}{\alpha + 2\lambda^*}y + \frac{2\lambda^*}{\alpha + 2\lambda^*}s\right)-s\right\|^2 = r.
\]
Upon the transformation $t^* = \frac{\alpha}{\alpha + 2\lambda^*}$, the above equation becomes piecewise linear quadratic and can be solved efficiently by a standard root-finding procedure.

In passing, we note that a solution procedure for the subproblem that arises when SCP is applied to \eqref{lOr} can be derived similarly, where the subproblem takes the form
\begin{equation*}
  \begin{array}{rl}
    \min\limits_{x} & \|x\|_1- \left<\xi,x\right>\\
    {\rm s.t.} & \|x - s\|^2 \le r,
  \end{array}
\end{equation*}
for some $\xi$, $s\in \R^n$ and $r > 0$. We omit the details for brevity.

\end{document}